\documentclass[reqno]{amsart}

\usepackage{amssymb}

\usepackage[draft=false]{hyperref}
\usepackage[noabbrev, capitalise, nameinlink]{cleveref}                                      
\hypersetup{                                                         
  colorlinks=true, 
  linkcolor={blue},
    citecolor={black},
    urlcolor={blue!80!black},
    linktocpage=true                                                      
}

\usepackage[dvipsnames]{xcolor}
\usepackage{tikz}	
\usetikzlibrary{cd, matrix, arrows, decorations.markings, decorations.pathmorphing}

\newtheorem{theorem}{Theorem}[section]

\newtheorem{corollary}[theorem]{Corollary}

\theoremstyle{definition}
\newtheorem{definition}[theorem]{Definition}
\newtheorem{example}[theorem]{Example}
\newtheorem{question}[theorem]{Question}

\theoremstyle{remark}
\newtheorem{remark}[theorem]{Remark}

\numberwithin{equation}{section}

\newcommand{\N}{\mathbb{N}}
\newcommand{\Z}{\mathbb{Z}}
\newcommand{\Q}{\mathbb{Q}}
\newcommand{\R}{\mathbb{R}}
\newcommand{\T}{\mathbb{T}}
\newcommand{\K}{\mathbb{K}}

\newcommand{\OO}{\mathcal{O}}
\newcommand{\D}{\mathcal{D}}

\DeclareMathOperator{\Id}{Id}
\DeclareMathOperator{\BF}{BF}
\DeclareMathOperator{\sign}{sign}
\DeclareMathOperator{\Inf}{Inf}

\begin{document}

\title[Invertible and noninvertible symbolic dynamics]{Invertible and noninvertible symbolic dynamics and their C*-algebras}

\author[K.A.~Brix]{Kevin Aguyar Brix}
\address{School of Mathematics and Statistics, University of Glasgow, University Place, Glasgow G12 8QQ, United Kingdom}
\email{kabrix.math@fastmail.com}
\thanks{The author gratefully acknowledges a DFF-international postdoc grant (Case number 1025-00004B)}



\begin{abstract}
  This paper surveys the recent advances in the interactions between symbolic dynamics and C*-algebras.
  We explain how conjugacies and orbit equivalences of both two-sided (invertible) and one-sided (noninvertible) symbolic systems may be encoded into C*-algebras,
  and how the dynamical systems can be recovered from structure-preserving *-isomorphisms of C*-algebras.
  We have included many illustrative examples as well as open problems.
\end{abstract}

\maketitle

\setcounter{tocdepth}{1}
\tableofcontents

\section{Introduction}

This paper surveys some of the recent advances in the interactions between topological dynamical systems and C*-algebras
with an emphasis on symbolic dynamics.
The earliest examples of encoding dynamics into operators on Hilbert space include the group von Neumann algebra of Murray and von Neumann~\cite{Murray-Neumann1943},
and von Neumann's ergodic theorem~\cite{Neumann1931};
a myriad of examples later followed via the crossed product construction (broadly construed).
It was Cuntz and Krieger~\cite{Cuntz-Krieger1980} who established the first connections between C*-algebras and symbolic dynamics (viz. shifts of finite type),
and this provided an abundance of simple and purely infinite C*-algebras.
An immediate advantage of constructing C*-algebras (e.g. Cuntz--Krieger algebras) from dynamics
is that C*-features such as ideal structure, KMS structure, or K-theory are intimately connected to the irreducible components, the entropy,
and known invariants of the dynamical systems, respectively.

With this work, I hope to also show that this influence can go the other way: dynamics can benefit from operator algebras.
This is perhaps most starkly seen in Giordano, Putnam, and Skau's classification of Cantor minimal systems up to orbit equivalence~\cite{Giordano-Putnam-Skau1995, Giordano-Putnam-Skau1999},
their use of crossed product C*-algebras, and topological full groups (which now live a mature life of their own),
but also in Krieger's dimension groups that serve to completely characterise shifts of finite type up to shift equivalence~\cite{Krieger1980a,Krieger1980b},
as well as Huang's classification of two-component shifts of finite type up to flow equivalence using an idea of Cuntz from C*-algebras~\cite{Huang1994}. 

Symbolic systems~\cite{Lind-Marcus} are useful e.g.~in the analysis of Axiom A diffeomorphisms \cite{Bowen1978}, mathematical linguistics (regular languages and sofic shifts),
and coding theory (e.g. digitalising analogue signals).
In this paper we focus on general shift spaces although there will be an emphasis on shifts of finite type 
because they are quite well understood and provide inspiration for the more general systems.
We are interested in classification up to e.g. conjugacy or orbit equivalences, and how this is reflected in the C*-algebras.
Invariably, this leads to a study of the fine-structure of C*-algebras (e.g. diagonal subalgebras and gauge actions) and structure-preserving *-isomorphisms between them.
For graphs, this approach is prevalent in the study of how geometrical moves are reflected in *-isomorphisms of graph C*-algebras~\cite{ERRS2016,Eilers-Ruiz}.

During the time of writing this paper, I gave talks at Queen's University Belfast and University of Cardiff on related topics,
and I thank Ying-Fen Lin and Ulrich Pennig for the kind invitations.
The fact that so many subtly different equivalence relations on dynamical systems appear in this topic presents a challenge when giving a talk.
It is my hope that this work may serve as a gentle introduction to the broader picture of how symbolic dynamical system and C*-algebras are related.
Most of the material covered here is established in the literature, and plenty of references are provided that the interested reader may consult,
although a few observations (especially in~\cref{sec:discrepancy}) seem to not have appeared before.

I have tried to keep the exposition short and, regrettably, there are many (interesting) aspects we cannot address here.
This includes graphs (there is a plethora of literature on directed graphs and their C*-algebras, see~\cite{Raeburn2005} and its references),
Leavitt path algebras (the algebraic counterpart of graph C*-algebras),
multidimensional dynamics, and
Smale spaces. 
Topological full groups will briefly be mentioned in~\cref{sec:continuous-orbit-equivalence} on continuous orbit equivalence.

What we will be able to discuss are the topics listed in the table of content
initiated by a preliminary section that covers the basics of symbolic dynamics and associated C*-algebras.
Two things to emphasise.
First, there is a tension between two-sided (invertible) systems, which seem to be of most interest to dynamicists,
and one-sided (noninvertible) systems that seem more natural from a C*-algebraic perspective.
For example,~\cite{Lind-Marcus} only dedidates a few pages to one-sided shifts and the new addendum to the book does not discuss them at all;
I hope to convey that one-sided dynamics can be just as interesting as its invertible sibling.
Second, there is an unfortunate discrepancy in the notion of one-sided eventual conjugacy and this is addressed in~\cref{sec:discrepancy}.

\section{Preliminaries} \label{sec:prelim}

Let $\Z$ be the integers, let $\N$ be the nonnegative integers (including $0$), and let $\N_+$ be the positive integers.
An $\N$-matrix $A$ is a matrix all of whose entries are in $\N$, and if $A$ is square, we let $|A|$ denote the dimension of $A$.
All matrices are finite-dimensional.

\subsection{Symbolic dynamics}
For a general and thorough introduction to classical theory of symbolic dynamics, we refer the reader to the excellent textbooks~\cite{Lind-Marcus} and~\cite{Kitchens}.
Although the two-sided systems consisting of bi-infinite sequences seem to be more popular among dynamicists we shall here emphasise the one-sided systems consisting of right-infinite sequences
as they fit nicer with the perspectives on C*-algebras.

Let $N\in \N_+$ and consider the set of \emph{symbols} $\{0,\ldots,N-1\}$ as a discrete space.
The \emph{full one-sided $N$-shift} is the compact metric space $\{0,\ldots,N-1\}^\N$ (equipped with the product topology) of right-infinite sequences 
together with the \emph{shift operation} $\sigma$ given by $\sigma(x)_i = x_{i+1}$ for all $i\in \N$ and $x\in \{0,\ldots,N-1\}^\N$.
A metric $d$ on $\{0,\ldots,N-1\}^\N$ may be given as $d(x,y) = 2^{-\min\{k\in \N : x_k \neq y_k\}}$, for distinct points $x,y\in \{0,\ldots,N-1\}^\N$.
The shift is a local homeomorphism (open and locally injective) and positively expansive 
in the sense that there is $\varepsilon > 0$ such that if $x,y\in \{0,\ldots,N-1\}^\N$ are distinct, 
then
\[
  d(\sigma^k(x), \sigma^k(y)) > \varepsilon.
\]
for some $k\in \N$.

If $x\in \{0,\ldots,N-1\}^\N$, then we write $x = x_{[0,\infty)} = x_0 x_1 \cdots$ where $x_i \in \{0,\ldots,N-1\}$,
and for integers $0\leq i \leq j$, we have $x_{[i,j]} = x_i x_{i+1} \cdots x_j$ is a finite word in the symbols $\{0,\ldots,N-1\}$
(similarly, $x_{[i,j)}$ and $x_{(i,j]}$ will have the obvious meaning).
We say the finite word $x_{[i,j}$ occurs in $x$, and we allow for a unique empty word $\epsilon$ that occurs in every sequence.

A \emph{one-sided subshift} (or a \emph{one-sided shift space}) is a closed subset $X$ of $\{0,\ldots,N-1\}^\N$
that is shift-invariant in the sense that $\sigma(X) \subset X$ (we do not assume equality here).
The restriction $\sigma_X$ of the shift to $X$ need not be open 
but it is always positively expansive, in particular it is locally injective.
The language $L(X)$ of a subshift $(\sigma, X)$ is all the finite words that occur in some $x\in X$,
and this in fact determines the whole subshift, cf.~\cite[Proposition 1.3.4]{Lind-Marcus}.
Moreover, a basis for the topology on $X$ is given by the cylinder sets
\[
  Z(\mu) = \{ x\in X : x_{[0,|\mu|)} = \mu \},
\]
where $\mu$ is a word that occurs in $X$.

A \emph{two-sided subshift} is a closed and shift-invariant subset $\underline{X} \subset \{0,\ldots,N-1\}^\Z$.
We use the notation $\underline{X}$ to emphasise that it is the space of bi-infinite sequences
although we use the same notation $\sigma$ for the shift; 
it should always be clear from the context what we mean.
The above notation also extends to the two-sided setting.
There is a natural continuous shift-commuting surjection (a \emph{factor map}) $\rho\colon \underline{X} \to \{0,\ldots,N-1\}^\N$ given by $\rho(x) = x_{[0,\infty)}$ for all $x\in \underline{X}$,
and the image $X$ is a one-sided subshift for which the shift is surjective.
Conversely, if $X$ is a one-sided shift for which $\sigma_X$ is surjective, we can recover the two-sided subshift $(\sigma_X, \underline{X})$
as the \emph{natural extension} given as the projective limit $\varprojlim_i(\sigma, X)$.

Two subshifts $(\sigma_X, X)$ and $(\sigma_Y, Y)$ are \emph{conjugate} if there exists a homeomorphism $h\colon X\to Y$ satisfying $h\circ \sigma_X = \sigma_Y\circ h$.

\begin{definition}
  A \emph{shift of finite type} is a subshift $(\sigma, X)$ for which there exists a finite set of words (the forbidden words) such that
  $X$ consists of all sequences that do not contain any forbidden words.
\end{definition}

\begin{example}\label{ex:SFT}
  The full $2$-shift may be represented by the directed graph
  
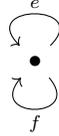
\[
    \begin{tikzcd}
      \bullet \arrow[loop, looseness=5,"e" above] \arrow[loop, looseness=5,, out=310, in=230,"f"] 
    \end{tikzcd} 
  \]
  The space of right-infinite edge paths is $\{e,f\}^\N$ and we may equip this with the shift operation.
  The golden mean graph is given as
  
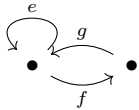
\[
  \begin{tikzcd}
    \bullet \arrow[loop, looseness=5, "e" above] \arrow[r, bend right, "f" below] & \bullet \arrow[l, bend right, "g" above]
  \end{tikzcd}
  \]
  and its space of right-infinite paths is contained in $\{e,f,g\}^\N$, but the words $\{ff, fe, gg, eg\}$ are not allowed.
  This is a finite list of words that are forbidden, so the path space of the graph represents a shift of finite type.
\end{example}

The shift on the path space of a directed graph (in particular for a shift of finite type) is a local homeomorphism.
Parry~\cite[Theorem 1]{Parry1966} shows that a subshift is a local homeomorphisms if and only if it is of finite type. 

\begin{remark}
Any shift of finite type may be represented (up to conjugacy) by a finite directed graph,~\cite[Chapter 2]{Lind-Marcus}.
As such the system is equivalently determined by the adjacency matrix of the graph, so we shall frequently denote a shift of finite type as $X_A$ 
where $A$ is the adjacency matrix of the graph that represents the system.
\end{remark}

\begin{example}\label{ex:even-shift}
  Consider the labelled graph
   
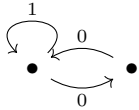
\[
  \begin{tikzcd}
    \bullet \arrow[loop, looseness=5, "1" above] \arrow[r, bend right, "0" below] & \bullet \arrow[l, bend right, "0" above]
  \end{tikzcd}
  \]
  with the labelled path space contained in $\{0,1\}^\N$.
  Any words of the form $1 0^{2n+1} 1$ cannot occur in any right-infinite path on the labelled graph,
  and this set of forbidden words cannot be reduced to a finite set.
  The shift is therefore not of finite type.
  It is usually called the \emph{even shift}; only an even number of $0$s is allowed between any two $1$s.
  Note that there is an obvious continuous and shift-commuting map (a factor map) from the path space of the golden mean graph of~\cref{ex:SFT} to the even shift.
  This is an example of a sofic shift.
\end{example}

The even shift shows that shifts of finite type have the intrinsic flaw of not being closed under factor maps (\emph{an unfair feature} in the words of Fischer). 
The smallest class containing shifts of finite type that is also closed under images is the sofic shifts
which were coined and studied by Weiss in~\cite{Weiss1973}.

\begin{definition}
  A subshift $(\sigma_X, X)$ is \emph{sofic} if it is the image of a continuous and shift-commuting map from a shift of finite type.
\end{definition}

\begin{remark}
  Any sofic shift may be represented (up to conjugacy) by a finite \emph{labelled} directed graph,~\cite[Chapter 3]{Lind-Marcus}.
  The underlying graph will then represent a shift of finite type from which there exists a factor map onto the sofic shift.
  Nasu~\cite{Nasu1986} has described a representation matrix (in non-commutative variables) that describe the labelled graph
  and this is a very useful tool to generalise many perspectives from the finite type case.
  We shall not discuss this further here.
\end{remark}

The representation of a sofic shifts is not unique,
but Fischer~\cite{Fischer1975} showed that any \emph{irreducible} sofic shift admits a unique labelled graph representation that is minimal and right-resolving 
(every vertex emits only edges with distinct labels).
This is not the case for general sofic shifts.
We now refer to this as the \emph{Fischer graph} or the \emph{Fischer cover}.
The Fischer graph of the even shift is depicted in~\cref{ex:even-shift}.

Krieger~\cite{Krieger1984} later associated to \emph{any} sofic shift a cover that is canonical
in the sense that any conjugacy at the level of sofic systems lifts to a unique conjugacy on the covers.
The Fischer graph is also canonical and two graphs need not coincide. 

\begin{definition}
  Let $(\sigma_X, X)$ be a one-sided subshift over the symbol set $\{1,\ldots,N\}$.
  The \emph{past set} of $x\in X$ is the set $P(x) = \{\nu \in L(X) : \nu x \in X\}$
  The \emph{(left) Krieger graph} of $(\sigma_X, X)$ is a labelled graph whose vertices is the collection of past sets,
  and there is an edge from $P$ to $P'$ with label $a\in \{1,\ldots,N\}$ if there exists $x\in X$ such that 
  $P = P(ax)$ and $P' = P(x)$.
\end{definition}

\begin{example}\label{ex:Krieger-cover}
  The labelled graph
   
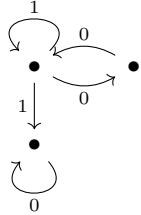
\[
  \begin{tikzcd}
    \bullet \arrow[loop, looseness=5, "1" above] \arrow[r, bend right, "0" below] \arrow[d, "1" left] & \bullet \arrow[l, bend right, "0" above]  \\
    \bullet  \arrow[loop, looseness=5,, out=310, in=230,"0" below]
  \end{tikzcd}
  \]
  is the Krieger cover of the even shift. 
  There are three distinct past sets e.g. represented by $0^\infty$, $10^\infty$, and $010^\infty$:
  \begin{align*}
    P(10^\infty)  &= \{ 0, 1, 00, 01, 11, \ldots\} \\
    P(010^\infty) &= \{ 0, 00, 10, \ldots\} \\
    P(0^\infty)   &= \{ 0, 1, 00, 01, 10, 11, \ldots\}.
  \end{align*}
  The reader may verify that the above graph is indeed the Krieger graph of the even shift.
  
  Observe that if $\mu$ is a word in the even shift that contains $1$ and if $\nu$ and $\omega$ are words such that $\nu \mu$ and $\mu \omega$ are allowed,
  then $\nu \mu \omega$ is also allowed. 
  However, if $\mu = 000$, then we can choose $\nu = 01$ and $\omega = 10$ and observe that both $\nu \mu$ and $\mu \omega$ are allowed 
  but $\nu \mu \omega = 01 000 10$ is not allowed. 
  A similar argument works for any word consisting only of $0$s.

  Every word occuring in $0^\infty$ is thus not synchronising, 
  and if we remove this part of the graph we obtain the Fischer graph of the even shift in~\cref{ex:even-shift}.
  This is no coincidence.
\end{example}

A word $\mu$ in a subshift $(\sigma_x, X)$ is \emph{intrinsically synchronising} if whenever $\nu$ and $\omega$ are words in $X$ and $\nu \mu$ and $\mu \omega$
then $\nu \mu \omega$ is also allowed.
In the even shift, a word is intrinsically synchronising if and only if it contains a $1$.
A right-infinite sequence $x\in X$ is synchronising if it contains an intrinsically synchronising word.
Krieger~\cite{Krieger1984} showed that the Fischer graph arises as the subgraph of the Krieger graph
whose vertices are the past sets of synchronising elements.

Below we mention a class of subshifts that is very different from the sofic shifts.

\begin{example}
  Let $R_\alpha$ be the rigid rotation on $\R/\Z$ by an irrational parameter $\alpha\in (0,1)\setminus \Q$.
  Define a coding map $\nu\colon \R/\Z \to \{0,1\}$ by $\nu(t) = 1$ if $t\in [0,\alpha)$, and $\nu(t) = 0$ if $t\in (\alpha, 1]$.
  The two-sided Sturmian shift is then the closure of $\{ \nu(R^i_\alpha(t))_{i\in \Z} : t\in \R/\Z\}$ in $\{0,1\}^\Z$.
  This is an example of a minimal homeomorphism on the Cantor space, cf.~\cref{sec:CMS},
  it has no periodic points and zero entropy, cf.~\cite[Section 13.7]{Lind-Marcus}.
  A one-sided Sturmian shift is defined analogously.
\end{example}

\subsection{C*-algebras}
Throughout the paper, we let $\K$ denote the compact operators on separable Hilbert space
and $c_0$ the commutative subalgebra of diagonal operators.

Let $A$ be an irreducible and nonpermutation $\{0,1\}$-matrix of dimension $n$.
Cuntz and Krieger~\cite{Cuntz-Krieger1980} define a C*-algebra $\OO_A$ (today known as the \emph{Cuntz--Krieger algebra}) 
as \emph{the} C*-algebra generated by $n$ partial isometries $s_1,\ldots,s_n$ satisfying the relation
\begin{equation} \label{eq:Cuntz-Krieger-family}
  1 = \sum_{i=1}^n s_i s_i^*, \qquad \textrm{and} \qquad   s_j^* s_j = \sum_{i=1}^n A(i,j) s_is_i^*
\end{equation}
for all $j=1,\ldots,n$.
Really they assume $A$ satisfy a technical condition (I) but irreducible and nonpermutation matrices are examples of this.
They show that $\OO_A$ is simple (it contains no nontrivial closed two-sided ideals),
and that it is defined independently of the specific Hilbert space on which the partial isometries act (up to canonical *-isomorphism), so the above \emph{the} is justified.
They provide a large class of examples of \emph{purely infinite} C*-algebras, cf.~\cite{Cuntz1981a}.
When the dimension of $A$ is $n\geq 2$ and every entry of $A$ is $1$ then the partial isometries are honest isometries, 
and $\OO_A$ is canonically isomorphic to Cuntz' algebra $\OO_n$, cf.~\cite{Cuntz1977}.
A diagonal subalgebra $\D_A$ of $\OO_A$ is generated by projections of the form $s_\mu s_\mu^*$ where $\mu = \mu_1 \cdots \mu_m$ is a finite word in $X_A$ and $s_\mu = s_{\mu_1} \cdots s_{\mu_m}$;
$\OO_A$ also admits a \emph{gauge action} $\gamma^A$ of $\T$ determined by $\gamma_z^A(s_i) = z s_i$ for all $z\in \T$ and generators $s_i$.

When $A$ is not irreducible the C*-algebras generated by similar relations is not simple and will generally depend on the Hilbert space on which the generators act.
It is however possible to define $\OO_A$ as a universal C*-algebra subject to \emph{Cuntz--Krieger families}~\cite[Section2]{anHuef-Raeburn1997}:
a Cuntz--Krieger family for $A$ is a collection of partial isometries $\{S_i\}_i$ satisfying~\labelcref{eq:Cuntz-Krieger-family},
and the C*-algebra $\OO_A$ is universal if it is generated by a Cuntz--Krieger family $\{s_i\}_i$,
and for any other family $\{S_i\}_i$ on a Hilbert space $H$ there is a *-representation $\OO_A \to B(H)$ sending $s_i \mapsto S_i$, for all $i$.

The construction of C*-algebras was extended to graphs in~\cite{Enomoto-Watatani1980, Kumjian-Pask-Raeburn-Renault1997}
where the latter used topological groupoids whose unit space is the right-infinite paths on the graphs.
This path space admits a shift operation which is a local homeomorphism.
We shall not dig deeper into graph C*-algebras here but simply refer the interested reader to Raeburn's monograph~\cite{Raeburn2005}
which constructs universal C*-algebras from generators and relations.

How can these constructions of C*-algebras be generalised to general subshifts e.g. sofic shifts?

Matsumoto was the first to associate C*-algebras to general subshifts in~\cite{Matsumoto1997} and subsequent papers using a Fock space construction.
Unfortunately, there was a mistake in a technical lemma which caused some confusion (this is however nicely clarified in~\cite{Carlsen-Matsumoto2004}).

For sofic shifts (under some technical conditions), Carlsen~\cite{Carlsen2003} and Samuel \cite{Samuel1998} independently observed that Matsumoto's C*-algebra  
is isomorphic to the Cuntz--Krieger algebras of their Krieger graphs.
In fact, we may now \emph{define} the C*-algebra of a sofic shift to be the Cuntz--Krieger algebra of its Krieger cover.
Similarly, we can define a C*-algebra of an irreducible sofic shift to be the Cuntz--Krieger algebra of its Fischer cover;
this will in general be a quotient of the former C*-algebra.

The approach in~\cite{Brix-Carlsen2020b} is to construct a canonical cover $(\sigma_{\tilde{X}}, \tilde{X})$ to any subshift $(\sigma_X, X)$ that generalises the path space of the Krieger graph.
This is a local homeomorphism from which we can construct an \'etale groupoid and therefore a C*-algebra $\OO_X$. 
The cover is useful because the shift on a general subshift is usually not a local homeomorphism (only if it is of finite type),
so the associated groupoid will usually not be \'etale.
Alternatively, Matsumoto~\cite{Matsumoto2020-normal} has constructed C*-algebras of so-called normal subshifts that generalise the Fischer cover.

The covers, groupoids, and C*-algebras of Sturmian shifts were studied in~\cite{Brix2021}.

\begin{remark}
We mention here briefly the advantage of \'etale groupoids.
Conceptually, groupoids can encode the orbit structure of a dynamical system and they should be thought of as the non-commutative orbit space.
The interested reader should consult~\cite{Sims-notes} for a friendly introduction to the topic
(see also~\cite{Renault-thesis} for the source of everything groupoid C*-algebras).

A local homeomorphism $T\colon X\to X$ on a locally compact Hausdorff space $X$
determines a groupoid
\[
  G_T = \bigcup_{m,n\in \N} \{ (x, m-n, y)\in X\times \Z \times X : T^m x = T^n y \}
\]
with unit space $G^{(0)} = \{ (x,0,x) : x\in X\}$ canonically identified with $X$,
and range and source maps $r(x,p,y) = x$ and $s(x,p,y) = y$ for all $(x,p,y)\in G_\sigma$.
A pair of elements $(x,p,y)$ and $(y',q,z)$ are composable precisely if $y=y'$ in which case $(x,p,y)(y,q,z) = (x,p+q,z)$,
and inversion is given as $(x,p,y)^{-1} = (y,-p,x)$.
Under a suitable topology inherited from $X$, $G_T$ is a locally compact Hausdorff, \'etale and amenable groupoid.
As such it defines groupoid C*-algebra $C^*(T) = C^*(G_T)$ which contains the functions on the unit space $C_0(X)$ (the \emph{diagonal}),
and the groupoid homomorphism $c_\sigma\colon G_T \to \Z$ given by $c_\sigma(x,p,y) = p$ for all $(x,p,y)\in G_\sigma$,
induces a \emph{gauge action} $\gamma^T$ of $\T$ on $C^*(G_T)$.
When $T = \sigma_A$ is a shift of finite type, then $C^*(G_{\sigma_A})$ is canonically isomorphic to $\OO_A$
in a way that is diagonal-preserving and gauge-equivariant.

A wonderful theorem of Renault~\cite{Renault2008} allows for a reconstruction of a topological groupoid (up to isomorphism) from its C*-algebra with diagonal subalgebra,
see also~\cite{Renault-thesis,Kumjian1986}.
Below we state a generalised version of the result from~\cite[Section 3]{CRST}.
The technical conditions are satisfied for all examples relevant to this paper and we shall not elaborate on them here.

\begin{theorem} \label{thm:Renault-reconstruction}
  Let $G_1$ and $G_2$ be a second-countable, Hausdorff, \'etale groupoids whose interior of isotropy is abelian and torsion-free.
  If there exists a *-isomorphism $\varphi\colon C^*_r(G_1) \to C^*_r(G_2)$ satisfying $\varphi(C_0(G_1^{(0)})) = C_0(G_2^{(0)})$,
  then $G_1$ and $G_2$ are isomorphic as topological groupoids.
\end{theorem}
This is an immensely useful tool in questions related to dynamical systems and their C*-algebras.
Many variations on Renault's reconstruction theory are now available in the literature but Renault's paper is still worth studying today.
\end{remark}

\section{Two-sided conjugacy} 
When are two dynamical systems the same? 
We can illustrate the immense complexity of this question by asking when two-sided shifts of finite type are conjugate.
This fundamental problem of symbolic dynacmis has still not been satisfactorially solved,
and the question of whether there exists a decision algorithm to determine conjugacy is still open.

We cannot address this problem without mentioning the efforts of Williams in the seminal work~\cite{Williams1973}.
In it, he addressed a question of Bowen (are shifts of finite type determined by their zeta function?)
and provided an algebraic classification of two-sided shifts of finite type in terms of $\N$-matrices; this is the notion of strong shift equivalence.

\subsection{Strong shift equivalence}
Let us start with an example.

\begin{example}[In-splitting] \label{ex:in-split}
  \label{ex:in-split-golden-mean}
  Consider the graph below with adjacency matrix $A$
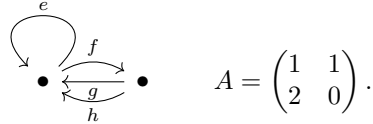
  \[
  \begin{tikzcd}
    \bullet \arrow[loop, "e" above] \arrow[r, bend left, "f" above] & \bullet \arrow[l, "g" below] \arrow[l, bend left, "h" below] 
  \end{tikzcd} \qquad
  A = 
  \begin{pmatrix}
    1 & 1 \\
    2 & 0
  \end{pmatrix}.
  \]
  The left-most vertex $v$ has three edges going \emph{in} ($e$, $g$ and $h$),
  and we split $v$ into two vertices (aligned on top of each other below) 
  and distribute the edges $e$ and $h$ to top copy of $v$ and $g$ to the bottom copy;
  the outgoing edges $e$ and $f$ are copied so there is one emitted from each of the two copies of $v$.
  We then get the graph below with adjacency matrix $B$
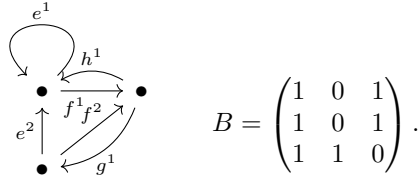
  \[
    \begin{tikzcd}
      \bullet \arrow[loop, "e^1" above] \arrow[r, "f^1" below left] & \bullet \arrow[l, bend right, "h^1" above] \arrow[dl, bend left, "g^1" below]\\
      \bullet \arrow[u, "e^2" left] \arrow[ur, "f^2" above]
    \end{tikzcd} \qquad
    B = 
    \begin{pmatrix}
      1 & 0 & 1 \\
      1 & 0 & 1 \\
      1 & 1 & 0
    \end{pmatrix}.
  \]
  Note that the loop $e$ is both incoming and outgoing for $v$ which is why we also have two copies.
  The two-sided shifts are conjugate, and it is possible to write down an explicit conjugacy.

  Alternatively, the in-split can be encoded into the two matrices 
  \[
    R = 
    \begin{pmatrix}
      1 & 0 \\
      1 & 0 \\
      1 & 1
    \end{pmatrix}, \qquad \textrm{and} \qquad
    S = 
    \begin{pmatrix}
      1 & 0 & 1 \\
      0 & 1 & 0
    \end{pmatrix},
  \]
  for which we have $B = RS$ and $SR = A$.
  This is an example of an (elementary) strong shift equivalent between $A$ and $B$.
\end{example}

Based on such observations, Williams introduced strong shift equivalence.

\begin{definition}
  A pair of square $\N$-matrices $A$ and $B$ are \emph{elementary strong shift equivalent}
  if there are rectangular $\N$-matrices $R$ and $S$ such that $A = RS$ and $SR = B$,
  and \emph{strong shift equivalent} if there are $\ell \in \N$ and matrices $A = C_0,\ldots,C_\ell = B$ such that 
  $C_i$ and $C_{i-1}$ are elementary strong shift equivalent for all $i = 1,\ldots,\ell$.
\end{definition}

The classification of two-sided shifts of finite type is from~\cite[Theorem A]{Williams1973}.
\begin{theorem} 
  A pair of two-sided shifts of finite type $(\sigma_A, \underline{X}_A)$ and $(\sigma_B, \underline{X}_B)$ are conjugate 
  if and only if the matrices $A$ and $B$ are strong shift equivalent.
\end{theorem}

For the general definition and notation of in-splits, we refer the reader to~\cite[Section 2.4]{Lind-Marcus} (or~\cite[Section 4]{Bates-Pask2004} for general graphs).
The proof of the classification theorem (see e.g.~\cite[Proof of Theorem 7.2.7]{Lind-Marcus}) is really a matter of decomposing an arbitrary conjugacy 
into a composition of elementary conjugacies coming from splittings such as the in-split above
and the dual notion of out-splits (see~\cref{ex:out-split} below) which also produces a strong shift equivalence.
As such, the theorem says that conjugacies can only arise by a sequence of splittings.

\begin{example}\label{ex:SE-for-all-k}
  Let $k\in \N$ and consider the matrices 
  \[
    A =
    \begin{pmatrix}
      1 & k \\
      k-1 & 1 
    \end{pmatrix}, \qquad \textrm{and} \qquad
    B = 
    \begin{pmatrix}
      1 & k(k-1) \\
      1 & 1 
    \end{pmatrix}.
  \]
  For $k=3$ the matrices are strong shift equivalent (in seven steps, cf.~\cite[Example 7.3.12]{Lind-Marcus}), but for $k\geq 4$ this is still an open problem.
  This example shows that the complexity of strong shift equivalence is not just in the dimensions of the adjacency matrices
  but in the number of elementary strong shift equivalences connecting the matrices as well as the sizes of these connecting matrices.
\end{example}

\begin{remark}
  Nasu~\cite{Nasu1986} (see also~\cite{Hamachi-Nasu1988}) has generalised the notion of strong shift equivalent to cover so-called representation matrices of sofic shifts.
  The paper shows that strong shift equivalence for these representation matrices coincides with conjugacy of the two-sided sofic systems. 
\end{remark}

\subsection{Shift equivalence} \label{sec:shift-equivalence}
Seeing that strong shift equivalence is hard to determine, Williams also introduced shift equivalence as an a priori weaker algebraic relation that seems easier to compute.

\begin{definition}\label{def:shift-equivalence}
  A pair of square $\N$-matrices $A$ and $B$ are \emph{shift equivalent}
  if there are $\ell \in \N_+$ and rectangular $\N$-matrices $R$ and $S$ such that 
  \begin{equation} 
    A^\ell = RS, \quad B^\ell = SR, \quad AR = RB, \quad BS = SA.
  \end{equation}
\end{definition}

It is not hard to verify the strong shift equivalence implies shift equivalence (which seems to justify the terminology).

If we consider the matrices as linear maps $A\colon Z^{|A|} \to Z^{|A|}$ and $B\colon Z^{|B|} \to Z^{|B|}$ 
then the shift equivalence relations are concisely expressed (in the case $\ell = 2$) in the following diagram
\[
\begin{tikzcd}
  \bullet \arrow[r, "A"] \arrow[d, "R" left] & \bullet \arrow[r, "A"] \arrow[d] & \bullet \arrow[r, "A" above] \arrow[d] & \cdots \arrow[r] & \Delta_A \\
  \bullet \arrow[r, "B" below] \arrow[urr, "S"] & \bullet \arrow[r, "B" below] \arrow[urr]& \bullet \arrow[r, "B" below]& \cdots \arrow[r] & \Delta_B \\
\end{tikzcd}
\]
where all the downwards vertical maps are $R$ and the upwards maps are $S$.
Here, $\Delta_A$ is the inductive limit which inherits the order structure $\Delta_A^+$ from the positive order on $\Z^{|A|}$,
and the matrix $A$ induces an automorphism $\delta_A$ on $\Delta_A$ which is essentially multiplication by $A$.
The matrices $R$ and $S$ induce group isomorphisms (that we also denote by $R$ and $S$) between $\Delta_A$ and $\Delta_B$,
and these isomorphisms preserve the order structure and intertwine the automorphisms $\delta_A$ and $\delta_B$.
This means that the module structure $(\Delta_A, \Delta_A^+, \delta_A)$ is an invariant of shift equivalence. 

In fact, the dimension module is a \emph{complete} invariant for shift equivalence.
This was proved by Krieger in~\cite{Krieger1980a,Krieger1980b} (although in the latter paper, the result is only stated for primitive matrices, 
see also~\cite[Theorem 7.5.8]{Lind-Marcus}).

\begin{theorem} 
  Let $A$ and $B$ be square $\N$-matrices.
  Then $(\sigma_A, \underline{X}_A)$ and $(\sigma_B, \underline{X}_B)$ are shift equivalent if and only if 
  the dimension triples $(\Delta_A, \Delta_A^+, \delta_A)$ and $(\Delta_B, \Delta_B^+, \delta_B)$ are isomorphic.
\end{theorem}

\begin{example}
  The matrices in~\cref{ex:SE-for-all-k} are shift equivalent for all $k\geq 3$.
  They are in fact similar (over $\Z$) and for primitive matrices this implies shift equivalent; 
  explicitly for $k=3$ the matrices 
  \[
    R = 
    \begin{pmatrix}
      8 & 3 \\
      1 & 16
    \end{pmatrix}, \qquad \textrm{and} \qquad
    S = 
    \begin{pmatrix}
      2 & 3 \\
      1 & 1
    \end{pmatrix}
  \]
  define a shift equivalence between $A$ and $B$.
\end{example}

Williams motivation for introducing shift equivalence was in part because it was \emph{more computable} than strong shift equivalence.
This intuition was formalised by Kim and Roush in~\cite{Kim-Roush1988} (see also~\cite{Kim-Roush1979})
where they showed that shift equivalence is decidable.
In~\cite[Theorem F]{Williams1973}, it is claimed that shift equivalence coincides with strong shift equivalence and therefore the classification problem for shifts of finite type should seemingly be resolved. 
Alas.
Parry found a mistake in the proof and the question of whether shift equivalence implies strong shift equivalence came to be known as the \emph{Williams problem},
the \emph{Shift equivalence problem}, or the \emph{Williams conjecture}.

Fast forward some twenty years, Kim and Roush exhibited a two-component counterexample to the Williams problem in~\cite{Kim-Roush1992},
and all hope to salvage the equality of shift equivalence and strong shift equivalence was extinguished a few years later
when they constructed an irreducible example.

\begin{example} \label{ex:Kim-Roush}
  The (in)famous primitive counterexample to the Williams problem by Kim and Roush~\cite[Section 7]{Kim-Roush1999} is given by the matrices 
  \[
    A =
    \begin{pmatrix}
    0 & 0 & 1 & 1 & 3 & 0 & 0 \\
    1 & 0 & 0 & 0 & 3 & 0 & 0 \\
    0 & 1 & 0 & 0 & 3 & 0 & 0 \\
    0 & 0 & 1 & 0 & 3 & 0 & 0 \\
    0 & 0 & 0 & 0 & 0 & 0 & 1 \\
    1 & 1 & 1 & 1 & 10 & 0 & 0 \\
    1 & 1 & 1 & 1 & 0 & 1 & 0 \\
  \end{pmatrix} \qquad \textrm{and} \qquad
    B = 
    \begin{pmatrix}
    0 & 0 & 1 & 1 & 3 & 0 & 0 \\
    1 & 0 & 0 & 0 & 0 & 0 & 0 \\
    0 & 1 & 0 & 0 & 0 & 0 & 0 \\
    0 & 0 & 1 & 0 & 0 & 0 & 0 \\
    0 & 0 & 0 & 0 & 0 & 0 & 1 \\
    4 & 5 & 6 & 3 & 10 & 0 & 0 \\
    4 & 5 & 6 & 3 & 0 & 1 & 0 \\
  \end{pmatrix}.
  \]
  In the paper, an explicit shift equivalence is given with \emph{integer} matrices,
  and for primitive matrices this implies shift equivalence with $\N$-matrices.
  Kim and Roush provide conditions that shift equivalent matrices must satisfy (based on the complicated \emph{sign-gyration} invariant) to not be strong shift equivalent
  (and thereby be counterexamples to the Williams problem).
  The bulk of the problem is then to find concrete examples satisfying those conditions; 
  indeed they remark that it \emph{was difficult for us}.
\end{example}
 
The paper of Kim and Roush~\cite{Kim-Roush1999} provide abstract conditions that counterexamples to the Williams problem must satisfy,
and the concrete examples of Kim and Roush show that the conditions are meaningful.
However, the precise relationship between shift equivalence and strong shift equivalence, in particular how one relation \emph{refines} the other,
seems to still be a mystery.

Since the classification of two-sided shifts of finite type is very complicated 
it is useful to have computable invariants to distinguish certain systems apart.
Bowen and Franks~\cite{Bowen-Franks} studied the two groups 
\begin{equation} \label{eq:Bowen-Franks}
  \ker(\Id - A), \qquad \textrm{and}\qquad   \BF(A) = \Z^{|A|}/ (\Id-A)\Z^{|A|},
\end{equation}
where $\Id -A$ is viewed as a map on $\Z^{|A|}$, and showed that they are invariants of conjugacy.
Today, we usually only refer to the cokernel as the \emph{Bowen--Franks group}.
The groups actually arose in their study of flow equivalence, cf.~\cref{sec:flow-equivalence}.
For example, the matrices in~\cref{ex:in-split} have Bowen--Franks group $\Z/2\Z$, and the Kim--Roush example~\cref{ex:Kim-Roush} have Bowen--Franks group isomorphic to $\Z/99 \Z$.

\subsection{C*-algebras} \label{sec:two-sided-conjugacy-C*}
Cuntz and Krieger introduced their C*-algebras from irreducible shifts of finite type (actually satisfying a technical condition (I)),
and they were already aware~\cite[Theorem 4.1]{Cuntz-Krieger1980} that the stabilised C*-algebra (together with the diagonal subalgebra) is an invariant of two-sided conjugacy.
In fact, they showed that the pair is invariant of flow equivalence, cf.~\cref{sec:flow-equivalence}.

Using the Pimsner--Voiculescu sequence, 
Cuntz~\cite[Proposition 3.1]{Cuntz1981b} identified the $K$-theory of the Cuntz--Krieger algebra $\OO_A$ with the invariants of Bowen and Franks (of the \emph{transposed} matrix $A^t$):
\[
  K_0(\OO_A) \cong \BF(A^t) \qquad \textrm{and} \qquad K_1(\OO_A) \cong \ker(\Id - A^t).
\]
See also~\cite{Cuntz1981a} where $K_0(\OO_A)$ is almost identified.
Strictly speaking, this was done under the hypothesis of Cuntz and Krieger's condition (I), 
but the argument has later been generalised to e.g. graphs, cf.~\cite[Chapter 7]{Raeburn2005}.
This provides strong evidence that the C*-algebras actually remember a lot of the underlying dynamics. 

R\o rdam~\cite{Rordam1995} later classified the simple Cuntz--Krieger algebras (those $\OO_A$ for which $A$ is irreducible and nonpermutation) up to stable isomorphism
by their $K_0$, i.e. by the Bowen--Franks group.
This deep result involved $KK$-theory and Huang's work on flow equivalence~\cite{Huang1994}.
Cuntz then observed that simple Cuntz--Krieger algebras are classified up to *-isomorphism by the Bowen--Franks group together with the class of the unit.

Further evidence is provided by considering the fixed-point algebra of the gauge action $\gamma^A$ on the Cuntz--Krieger algebra $\OO_A$.
This is an approximately finite-dimensional (AF) C*-algebra,
the $K_0$-group of which is identified with the dimension group as an ordered group (again of the transposed matrix $A^t$).
In fact the equivariant $K$-theory of the gauge action coincides with Krieger's dimension triple (the underlying group is the $K_0$-group of the fixed-point algebras).

The theorem below is from Bratteli and Kishimoto~\cite[Corollary 1.5 and Theorem 4.3]{Bratteli-Kishimoto},
and its proof relies on classification of actions on C*-algebras from $K$-theoretic data.
The importance of the matrices being primitive is to ensure that the fixed-point algebra is simple with one-dimensional lattices of traces.

\begin{theorem}\label{thm:Bratteli-Kishimoto} 
  Let $A$ and $B$ be primitive matrices.
  They are shift equivalent if and only if there is a *-isomorphism 
  $\varphi\colon \OO_A\otimes \K \to \OO_B\otimes \K$ satisfying $\varphi\circ (\gamma^A\otimes \Id) = (\gamma^B\otimes \Id) \circ \varphi$.
\end{theorem}

\begin{remark}
  It is not known to which extend such a characterisation holds.
  The issue is really going \emph{from} shift equivalent matrices \emph{to} equivariantly *-isomorphic C*-algebras;
  the other direction follows from the identification of the $K$-theory with the dimension triple.
  Recently, Eilers and Szab\'o have announced that they can generalise the result to cover irreducible and nonpermutation matrices.
\end{remark}

Quite surprisingly, it is possible to characterise strong shift equivalence (or, equivalently, two-sided conjugacy of shifts of finite type) using C*-algebras.
Matsumoto~\cite[Theorem 2.18]{Matsumoto2017-circle} made progress on this for irreducible and nonpermutation matrices
and characterised it in terms of Morita equivalence of the triple $(\OO_A, C(X_A), \gamma^A)$, the Cuntz--Krieger algebra, its diagonal subalgebra, and the canonical gauge action.
Below we record a slightly different characterisation from~\cite[Theorem 5.1]{Carlsen-Rout2017};
this result has no irreducibility conditions.

\begin{theorem} \label{thm:Carlsen-Rout}
  Let $A$ and $B$ be square nonnegative $\N$-matrices with no zero rows or columns.
  The two-sided shifts $(\sigma_A, \underline{X}_A)$ and $(\sigma_B, \underline{X}_B)$ are conjugate  
  if and only if there exists a *-isomorphism $\varphi\colon \OO_A\otimes \K \to \OO_B\otimes \K$ satisfying
  $\varphi(C(X_A)\otimes c_0) = C(X_B)\otimes c_0$ and $\varphi\circ (\gamma^A_z\otimes \Id) = (\gamma^B_z\otimes \Id)\circ \varphi$ for all $z\in \T$.
\end{theorem}

A similar theorem holds for general subshifts (not necessarily of finite type) by~\cite[Theorem 7.5]{Brix-Carlsen2020b}.

\begin{remark}
It is curious that shift equivalence seems to be harder to encode into C*-algebras than strong shift equivalence.
Perhaps this follows from the fact that the former relation is not an orbit equivalence
and so groupoid techniques do not seem fit for the task.
\end{remark}

\section{One-sided conjugacy}
A conjugacy of one-sided shifts lifts to a conjugacy of their natural extensions (the correponding two-sided shifts)
but the converse need not be the case.
In fact, the conjugacy question for noninvertible systems turns out to be starkly different from that of the invertible systems.
In his influencial paper~\cite{Williams1973}, Williams devises an algorithm to solve the conjugacy problem for one-sided shifts of finite type
and observes \emph{almost as a by-product we obtain a classification of 1-sided shifts, up to topological conjugacy}.

\subsection{Williams' amalgamation algorithm}
The in-spliting process of~\cref{ex:in-split} has a dual notion called \emph{out-splitting}.

\begin{example}[Out-splitting] 
  \label{ex:out-split}
  Consider the graph below with adjacency matrix $A$
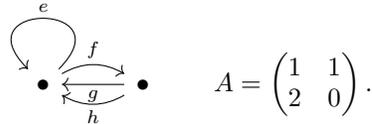
  \[
  \begin{tikzcd}
    \bullet \arrow[loop, "e" above] \arrow[r, bend left, "f" above] & \bullet \arrow[l, "g" below] \arrow[l, bend left, "h" below]
  \end{tikzcd} \qquad
  A = 
  \begin{pmatrix}
    1 & 1 \\
    2 & 0
  \end{pmatrix}.
  \]
  The left-most vertex $v$ has two edges going \emph{out} ($e$ and $f$),
  and we split $v$ into two vertices (aligned on top of each other below) 
  and distribute the two edges to each copy of $v$;
  the incoming edges $g$ and $h$ are then copied so there is one entering each of the two copies of $v$.
  We then get the graph below with adjacency matrix $B$
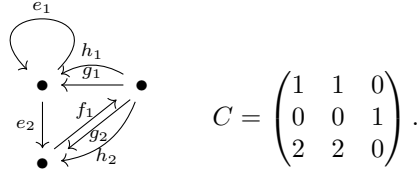
  \[
    \begin{tikzcd}
      \bullet \arrow[loop, "e_1" above] \arrow[d, "e_2" left] & \bullet \arrow[l, "g_1" above] \arrow[dl, transform canvas={xshift=0.6ex}, "g_2" below] 
      \arrow[l, bend right, "h_1" above]\arrow[dl, bend left, "h_2" below] \\
      \bullet \arrow[transform canvas={xshift=-0.6ex}, "f_1" above]{ur} &
    \end{tikzcd} \qquad
    C = 
    \begin{pmatrix}
      1 & 1 & 0 \\
      0 & 0 & 1 \\
      2 & 2 & 0
    \end{pmatrix}.
  \]
  Note that the loop $e$ is both incoming and outgoing for $v$ which is why we also have two copies.
  As in~\cref{ex:in-split} it is easy to write down matrices $R$ and $S$ that verify that $A$ and $C$ are strong shift equivalent.

  The reader may verify that an explicit one-sided conjugacy $h\colon X_A \to X_C$ is given as follows:
  any occurrence of $ee$ is mapped to $e^1$, $ef$ is mapped to $e^2$, $f$ is mapped to $f^1$, $ge$ is mapped to $g^1$, $gf$ is mapped to $g^2$,
  $he$ is mapped to $h^1$, and $hf$ is mapped to $h^2$.
  For example,
  \[
    h(eefgefhf \cdots ) = e^1 e^2 f^1 g^1 e^2 f^1 h^2 f^1 \cdots
  \]
\end{example}

For the general definition and notation of out-splits, we refer the reader to~\cite[Section 2.4]{Lind-Marcus} (or~\cite[Section 3]{Bates-Pask2004} for general graphs).
Any out-split produces conjugate graphs.

The inverse operation to an out-split is called an \emph{out-amalgamation} (here we will just say amalgamation) and naturally this operation also preserves the one-sided conjugacy class of the system.
The surprising fact is that these two operations exhaust all conjugacies of one-sided shifts of finite type.

At the level of matrices the amalgamation may be described as follows:
an \emph{amalgamation} of $A$ is obtained by collapsing identical columns of $A$ and adding up the corresponding rows entrywise.
An example is the matrix $C$ in~\cref{ex:out-split} in which the first two columns are identical.
By keeping only one of the columns and adding the first are second rows, we arrive at the matrix $A$. 
The dimension of the amalgamated matrix is then less than the dimension of $A$.
The \emph{total amalgamation} of a square $\N$-matrix $A$ is obtained by iteratively collapsing identical columns and adding corresponding rows
until the resulting matrix contains no identical columns.
Williams showed that this always exists (it is independent of the order of amalagamations) and that it is unique (up to permutation).
The classification for one-sided shifts of finite type is from~\cite[Theorem G]{Williams1973}. 

\begin{theorem} \label{thm:Williams-one-sided}
  Any square $\N$-matrix admits a unique total amalgamation (up to permutation),
  and if $A$ and $B$ are square $\N$-matrices, then their one-sided shifts are conjugate if and only if 
  the total amalgamations of $A$ and $B$ agree (up to permutation).
\end{theorem}

Observe that in~\cref{ex:out-split} above, $A$ has no identical columns, so it is its own total amalgamation.
However, the first two columns of $C$ are identical and the total amalgamation of $C$ coincides with $A$.
Conversely, both of the graphs in the in-split example (\cref{ex:in-split}) contain no identical columns, and they are distinct,
so this shows that the one-sided shifts $X_A$ and $X_B$ are not one-sided conjugate (though they are two-sided conjugate).

\begin{remark}
For graphs with finitely many vertices (but potentially infinitely many edges) a generalisation of Williams' algorithm works to show that conjugacy for such graphs is decidable.
This result will appear in forthcoming work~\cite{ABCEW}.
\end{remark}

\subsection{General C*-description for local homeomorphisms}
We saw in~\cref{sec:two-sided-conjugacy-C*} that Cuntz and Krieger showed that $\OO_A$ is an invariant of the two-sided shift of finite type $(\sigma_A, \underline{X}_A)$
\emph{up to stable isomorphism}.
Above we also saw that one-sided conjugacy is completely classified using out-splits and amalgamations. 
In~\cite{Bates-Pask2004}, Bates and Pask described how to generalise Williams' splitting operations to possibly infinite graphs,
and they observed that in-splits produce graph C*-algebras that are stably isomorphic
while out-splits produce graph C*-algebras that are honestly *-isomorphic.
From~\cref{thm:Williams-one-sided}, we then know that $\OO_A$ is really an invariant of \emph{one-sided conjugacy}.
This was in fact already observed by Cuntz and Krieger, cf.~\cite[Proposition 2.17]{Cuntz-Krieger1980},
and the point is also made explicit by Katsura~\cite{Katsura2009}.

The main result in~\cite{Brix-Carlsen2020a} is a converse to~\cite[Proposition 2.17]{Cuntz-Krieger1980}
in which one-sided conjugacy of shifts of finite type is characterised in terms of diagonal-preserving *-isomorphism of Cuntz--Krieger algebras 
that intertwine a certain completely positive map.
Working with irreducible shifts of finite type in~\cite{Matsumoto2022}, Matsumoto dispensed with the completely positive map and instead required that the *-isomorphism intertwine a whole family a gauge actions.
The following theorem from~\cite[Theorem 3.1]{Armstrong-Brix-Carlsen-Eilers} is the appropriate generalisation of Matsumoto's idea to general local homeomorphisms.
It illustrates an example of a connection between general dynamical systems in the form of local homeomorphisms and their C*-algebras
that is directly inspired by shifts of finite type.

\begin{theorem} 
  Let $T\colon X\to X$ and $S\colon Y\to Y$ be local homeomorphisms.
  Then $T$ and $S$ are conjugate if and only there is a *-isomorphism $\varphi\colon C^*(T) \to C^*(S)$ satisfying 
  $\varphi(C_0(X)) = C_0(Y)$ (implemented by a homeomorphism $h\colon X\to Y$), and
  \[
    \varphi\circ \gamma_t^{X,f} = \gamma_t^{Y,f\circ h^{-1}}\circ \varphi,
  \]
  for all $t\in \R$ and $f\in C_0(X, \R)$.
\end{theorem}

When the space is totally disconnected, we may replace $\R$ by $\Z$
(this applies e.g. to (infinite) directed graphs).

Below we specify it to general subshifts.
This has not appeared elsewhere and it complements Matsumoto's work on conjugacy of normal subshifts~\cite[Theorem 1.5]{Matsumoto2021} and~\cite[Theorem 4.4]{Brix-Carlsen2020b}.
We emphasise that the result below applies to all general subshifts.
This class of dynamical systems is not immediately covered by the above theorem because the action is not by local homeomorphisms.
We provide only a sketch of proof; the details are an adaptation of the proof of the above theorem to general subshifts as in~\cite{Brix-Carlsen2020b}.

\begin{corollary} 
  Let $(\sigma_X, X)$ and $(\sigma_Y, Y)$ be general one-sided shift spaces.
  The systems are conjugate if and only if there is a *-isomorphism $\varphi\colon \OO_X \to \OO_Y$ satisfying 
  $\varphi(C(X)) = C(Y)$ (implemented by a homeomorphism $h\colon X\to Y$) and 
  \begin{equation} \label{eq:subshifts-gauge-actions}
    \varphi\circ \gamma_z^{X,f\circ \pi_X} = \gamma_z^{Y, (f\circ h^{-1})\circ \pi_Y}\circ \varphi,
  \end{equation}
  for all $f\in C(X, \Z)$ and $z\in \T$.
  Here, $\pi_X\colon \tilde{X} \to X$ and $\pi_Y\colon \tilde{Y}\to Y$ are the canonical factor maps associated to the systems.
\end{corollary}

\begin{proof}
  [Sketch of proof]
  Assume there exist a homeomorphism $h\colon X \to Y$ and a *-isomorphism $\varphi\colon \OO_X \to \OO_Y$ satisfying 
  $\varphi(C(X)) = C(Y)$ with $\varphi(f) = f\circ h^{-1}$ for all $f\in C(X)$, and~\labelcref{eq:subshifts-gauge-actions}.
  By~\cite[Theorem 3.3]{Brix-Carlsen2020b}, we have $\varphi(D_X) = D_Y$,
  and this means there is a homeomorphism $\tilde{h}\colon \tilde{X}\to \tilde{Y}$ of the covers satisfying $h\circ \pi_X = \pi_Y\circ \tilde{h}$.
  Applying~\labelcref{eq:subshifts-gauge-actions} for $f=1$, it follows that $h$ and $\tilde{h}$ are one-sided eventual conjugacies (cf.~\cref{def:Matsumoto-eventual-conjugacy} below).
  Let $k\in \N$ be the constant associated to $h$ and $\tilde{h}$ according to the definition.

  Now, fix $x\in X$ and let $\tilde{x}\in \pi_X^{-1}(x) \subset \tilde{X}$.
  The hypothesis~\labelcref{eq:subshifts-gauge-actions} implies (this requires some work) that for any $f\in C(X,\Z)$
  \[
    f(x) = \sum_{i=0}^k f\circ h^{-1}\circ \pi_Y(\sigma^i_{\tilde{Y}}(\tilde{h}(\tilde{x}))) - \sum_{j=0}^k f\circ h^{-1}\circ \pi_Y(\sigma^j_{\tilde{Y}}(\tilde{h}(\sigma_{\tilde{X}}(\tilde{x})))).
  \]
  If we let $u = \sigma_Y(h(x))$ and $v = h(\sigma_X(x))$, then this simplifies to
  \[
    \sum_{i=0}^k f\circ h^{-1}(\sigma_Y^i(u)) = \sum_{j=0}^k f\circ h^{-1}(\sigma_Y^j(v))
  \]
  for all $f\in C(X,\Z)$.
  Now~\cite[Lemma 3.7]{Armstrong-Brix-Carlsen-Eilers} (this lemma is stated for systems of local homeomorphisms but this is not needed) 
  allows us to conclude that $u = v$ and so $h$ is a conjugacy.
\end{proof}

\section{Matsumoto's one-sided eventual conjugacy} \label{sec:Matsumoto-eventual-conjugacy}
The conjugacy problem for one-sided shifts of finite type is understood 
via out-splits and amalgamations, and there is a user-friendly decision procedure to determine when two such systems are conjugate. 
In this section, we discuss Matsumoto's notion of eventual conjugacy~\cite{Matsumoto2017-circle}, its relation to C*-algebras,
and how it relates to strong shift equivalence.
We shall be careful to call this eventual conjugacy \emph{in Matsumoto's sense} because there is an arguably more natural notion of eventual conjugacy 
(which we shall discuss in~\cref{sec:discrepancy}), and the two do not coincide.
Matsumoto's eventual conjugacy was first defined for irreducible shifts of finite type, but here we state it for local homeomorphisms.

\begin{definition}\label{def:Matsumoto-eventual-conjugacy}
  Let $T\colon X \to X$ and $S\colon Y\to Y$ be local homeomorphism.
  A homeomorphism $h\colon X \to Y$ is a \emph{Matsumoto eventual conjugacy} if there are continuous maps $k\colon X \to \N$ and $k'\colon Y\to \N$ satisfying
  \begin{align} 
    S^{k(x) + 1}\circ h(x) &= S^{k(x)}\circ h\circ T(x), \\
    T^{k'(y) + 1}\circ h^{-1}(y) &= T^{k'(y)}\circ h^{-1}\circ S(y),
  \end{align}
  for all $x\in X$ and $y\in Y$.
  When $X$ and $Y$ are compact, we may equivalently choose $k$ and $k'$ to be fixed nonnegative integers. 
\end{definition}

Note that $h$ is a conjugacy if we can choose $k = 0$.

\begin{example}\label{ex:Brix-Carlsen2020a}
The example below from~\cite{Brix-Carlsen2020a} shows that Matsumoto's eventual conjugacy is nontrivial in the sense that it does not coincide with one-sided conjugacy.
Let $X$ and $Y$ be the one-sided shift spaces determined by the graphs 
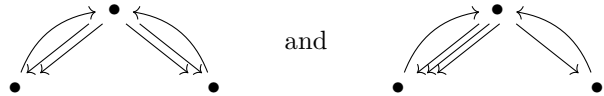
\[
  \begin{tikzcd}
    & \bullet \arrow[transform canvas={xshift=0.6ex}]{dl} \arrow[transform canvas={xshift=-0.6ex}]{dl} \arrow[transform canvas={xshift=0.6ex}]{dr} \arrow[transform canvas={xshift=-0.6ex}]{dr}& \\
    \bullet \arrow[ur, bend left] & & \bullet \arrow[ul, bend right]
  \end{tikzcd}
  \qquad \textrm{and} \qquad
  \begin{tikzcd}
    & \bullet \arrow{dl} \arrow[transform canvas={xshift=0.9ex}]{dl} \arrow[transform canvas={xshift=1.8ex}]{dl}  \arrow{dr}& \\
    \bullet \arrow[ur, bend left] & & \bullet \arrow[ul, bend right]
  \end{tikzcd}
\]
respectively.
A concrete eventual conjugacy is given as follows: 
label the edges pointing downwards in the first graph from left to right by $a,b,c,d$ and edges pointing upwards $e,f$;
similarly, label the edges of the second graph by $a',b',c',d',e',f'$.
Define $h\colon X \to Y$ by sending an infinite path $x = x_0x_1\cdots$ to $y = y_0y_1\cdots$
where $y_i = e'$ if $x_{i-1} = c$, and $y_i = x_i'$ otherwise. 
The reader may verify that this is a Matsumoto eventual conjugacy.
On the other hand, the two graphs have distinct total amalgamations which can easily be checked so the one-sided systems are not conjugate.
\end{example}

The following C*-algebraic characterisation has its origin in~\cite[Theorem 1.2]{Matsumoto2017-circle} for infinite irreducible shifts of finite type.
It was later developed for graphs in~\cite[Theorem 4.1]{Carlsen-Rout2017},
and found its final form for general local homeomorphisms in~\cite[Theorem 8.10]{CRST}.
The C*-algebras that appear are the groupoid C*-algebras built from the local homeomorphisms as in~\cref{sec:prelim};
for shifts of finite type they are just the Cuntz--Krieger algebras.

\begin{theorem} \label{thm:Matsumoto-eventual-conjugacy}
  Let $T\colon X\to X$ and $S\colon Y\to Y$ be local homeomorphisms.
  Then $T$ and $S$ are are Matsumoto eventually conjugate if and only if there is a *-isomorphism $\varphi\colon C^*(T) \to C^*(S)$ 
  satisfying $\varphi(C_0(X)) = C_0(Y)$ and $\varphi\circ \gamma^T_z = \gamma^S_z\circ \varphi$, for all $z\in \T$.
\end{theorem}

The proof encodes an eventual conjugacy into groupoids and then uses the reconstruction theory for topological groupoids with its canonical cocycle;
this is a variation on~\cref{thm:Renault-reconstruction}.
A similar result for general shift spaces can be found in~\cite[Theorem 5.3]{Brix-Carlsen2020b}.

\subsection{Balanced strong shift equivalence}
Now we discuss \emph{balanced strong shift equivalence} as introduced in~\cite[Section 5]{Brix2022}.
This is an equivalence relation on square $\N$-matrices which should be thought of as the one-sided analogue of Williams' strong shift equivalence.

The pair of graphs in~\cref{ex:Brix-Carlsen2020a} inspired a new move on graphs introduced in~\cite{Eilers-Ruiz},
the \emph{balanced in-split}: 
one graph $G$ produces two graphs $E$ and $F$ by performing two in-splits at $G$ in which the vertices are split into \emph{the same} number of new vertices,
but the incoming edges may be distributed \emph{differently}.
The pair of graphs $E$ and $F$ are then said to be a balanced in-split of $G$.

\begin{example}
Consider the graph
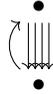
\[
  \begin{tikzcd}
    \bullet \arrow[transform canvas={xshift=0.3ex}]{d} \arrow[transform canvas={xshift=0.9ex}]{d} \arrow[transform canvas={xshift=-0.3ex}]{d} \arrow[transform canvas={xshift=-0.9ex}]{d} \\
    \bullet \arrow[u, bend left=50] 
  \end{tikzcd}
\]
The bottom vertex has four incoming edges, and the reader can verify that the two graphs in~\cref{ex:Brix-Carlsen2020a} arise as in-splits of this graph in two different ways:
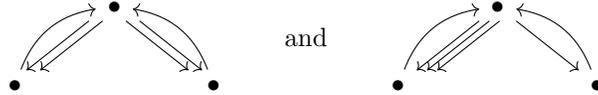
\[
  \begin{tikzcd}
    & \bullet \arrow[transform canvas={xshift=0.6ex}]{dl} \arrow[transform canvas={xshift=-0.6ex}]{dl} \arrow[transform canvas={xshift=0.6ex}]{dr} \arrow[transform canvas={xshift=-0.6ex}]{dr}& \\
    \bullet \arrow[ur, bend left] & & \bullet \arrow[ul, bend right]
  \end{tikzcd}
  \qquad \textrm{and} \qquad
  \begin{tikzcd}
    & \bullet \arrow{dl} \arrow[transform canvas={xshift=0.9ex}]{dl} \arrow[transform canvas={xshift=1.8ex}]{dl}  \arrow{dr}& \\
    \bullet \arrow[ur, bend left] & & \bullet \arrow[ul, bend right]
  \end{tikzcd}
\]
In both cases, the bottom vertex is split into two vertices but the four edges are distributed differently:
in the left-most graph the two vertices are given two edges each,
while in the right-most graph the edges are distributed with three edges to one vertex and a single edge to the other vertex.
As is always the case for an in-split, the outgoing edge is copied.
\end{example}

It is shown in~\cite[Corollary 3.7]{Eilers-Ruiz} that the balanced in-split induces a *-isomorphism of graph C*-algebras
that is both diagonal-preserving and gauge-equivariant.
Via~\cref{thm:Matsumoto-eventual-conjugacy}, this means that this move induces an eventual conjugacy in Matsumoto's sense on the one-sided shift spaces.
The paper~\cite{Brix2022} gives a direct proof of this fact and also shows the converse for finite graphs with no sinks: 
any eventual conjugacy is the result of balanced in-splits up to one-sided conjugacy.
Rephrasing the balanced in-split into the language of matrices we arrive at the definition of balanced strong shift equivalence.

\begin{definition}
  Let $A$ and $B$ be square $\N$-matrices.
  They are \emph{elementary balanced strong shift equivalent} if there are rectangular $\N$-matrices $R_1$, $R_2$, and $S$ such that
  \[
    A = SR_1, \qquad B = SR_2, \qquad R_1 S = R_2 S.
  \]
  The matrices are \emph{balanced strong shift equivalent} if there are conjugate copies $A'$ and $B'$ of $A$ and $B$, respectively,
  and $n\in \N_+$ and matrices $A' = A_0, \ldots,A_n = B'$ such that $A_{i-1}$ and $A_i$ are elementary balanced strong shift equivalent for all $i=1,\ldots,n$.
\end{definition}

The main result of~\cite{Brix2022} is then the following characterisation.

\begin{theorem} 
  Let $A$ and $B$ be square $\N$-matrices with no zero rows (the graphs have no sinks).
  The one-sided shifts $(\sigma_A, X_A)$ and $(\sigma_B, X_B)$ are Matsumoto eventually conjugate if and only if
  $A$ and $B$ are balanced strong shift equivalent.
\end{theorem}

Analogous to Williams' classification of two-sided conjugacies,
this result says that $(\sigma_A, X_A)$ and $(\sigma_B, X_B)$ are Matsumoto eventually conjugate if and only the underlying graphs
can be connected by out-splits and balanced in-splits.
The assumption that the graphs admit no sinks is merely a technical one;
it should be possible to prove a similar results for a large class of graphs.

Noting the connection to two-sided conjugacy and strong shift equivalence, we ask the obvious decidability question
which was also mentioned in~\cite{Brix2022}.

\begin{question}
  Is balanced strong shift equivalence decidable?
\end{question}

\section{A discrepancy in one-sided eventual conjugacy} \label{sec:discrepancy}
Williams introduced shift equivalence because it is tamer than strong shift equivalence and easier to determine.
This is an \emph{algebraic} relation on matrices, and the analogous \emph{dynamical} description was only later discovered to be two-sided eventual conjugacy:
all higher powers are two-sided conjugate.
One-sided eventual conjugacy (in the sense of having one-sided conjugate higher powers) was studied in depth in~\cite{Boyle-Fiebig-Fiebig},
and the authors show that this relation is decidable by a very user-friendly algorithm.
We observe in this section the unfortunate fact that this does not coincide with Matsumoto's eventual conjugacy from~\cref{sec:Matsumoto-eventual-conjugacy}.

We can in fact also make sense of a shift equivalence for one-sided shifts (we call this \emph{unital shift equivalence}),
and it turns out that this does not coincide with either of the two version of one-sided eventual conjugacy.
This will be elaborated on in future work~\cite{Brix-unitalSE}.

\subsection{Conjugate higher powers}
We first define the notion of having conjugate higher powers.
This is usually called eventual conjugacy.

\begin{definition}
  Let $\sigma_X\colon X\to X$ and $\sigma_Y\colon Y\to Y$ be dynamical systems (e.g. $\sigma_X$ and $\sigma_Y$ are both (local) homeomorphisms).
  The systems are \emph{eventually conjugate} 
  (or have \emph{conjugate higher powers})
  if the higher powers $\sigma_X^i$ and $\sigma_Y^i$ are conjugate for all but finitely many $i\in \N_+$.
\end{definition}

For shifts of finite type (one-sided or two-sided), the higher power $(\sigma_A^i, X_A)$ is canonically conjugate to $(\sigma_{A^i}, X_{A^i})$.

\begin{example}
  It is important that all sufficiently high powers are conjugate.
  The matrices
  \[
    A =
    \begin{pmatrix}
      0 & 2 \\
      2 & 0
    \end{pmatrix}\quad \textrm{and} \quad
    B =
    \begin{pmatrix}
      2 & 0 \\
      0 & 2
    \end{pmatrix}
  \]
  satisfy $A^{2n} = B^{2n}$ for all $n\in \N$ but one is irreducible while the other is reducible.
\end{example}

In the case of two-sided shifts of finite type, having conjugate higher powers coincides with shift equivalence, cf.~\cref{sec:shift-equivalence}.
One direction was observed by Williams~\cite{Williams1973}, and the converse was proved by~Kim and Roush~\cite{Kim-Roush1979} (see also~\cite[Theorem 7.5.15]{Lind-Marcus}).

\begin{theorem} \label{thm:SE-higher-powers}
  Let $(\sigma_A, \underline{X}_A)$ and $(\sigma_B, \underline{X}_B)$ be two-sided shifts of finite type.
  The matrices $A$ and $B$ are shift equivalent if and only if the systems have conjugate higher powers.
\end{theorem}

This result provides the dynamical analogue of the algebraic notion of shift equivalence.
A similar correspondence does not occur in the one-sided setting; this will be elaborated on elsewhere.

\begin{remark}
  A similar result applies to sofic shifts.
  Boyle and Krieger~\cite{Boyle-Krieger1988} introduce and study shift equivalence for two-sided sofic systems
  and show that it coincides with having conjugate higher powers. 
\end{remark}

In~\cite{Boyle-Fiebig-Fiebig}, Boyle, D.~Fiebig, and U.~Fiebig study one-sided shifts of finite type that have conjugate higher powers
(they refer to this as one-sided eventual conjugacy).
We present here the main result related to this notion.

Let $N(A)$ be the least positive integer satisfying $\textrm{rank} (A^n) = \textrm{rank} (A^{n + 1})$
which is at most the dimension of $A$.
The surprisingly nice characterisation below is from~\cite[Section 8]{Boyle-Fiebig-Fiebig}.

\begin{theorem} \label{thm:BFF}
  Let $A$ and $B$ be square $\N$-matrices and let $n \geq \max\{ N(A), N(B) \}$.
  The following are equivalent.
  \begin{enumerate}
    \item $A$ and $B$ determine one-sided shifts of finite type that have conjugate higher powers;
    \item the total amalgamations of $A^m$ and $B^m$ agree (up to the same permutation) for $m = n, n+1$; and
    \item the total amalgamations of $A^m$ and $B^m$ agree (up to the same permutation) for all $m\geq n$.
  \end{enumerate}
\end{theorem}

An immediate consequence of this theorem is that having conjugate higher powers for one-sided shifts of finite type is decidable.
Simply construct the total amalgamations of $A^m$ and $B^m$ and check whether they are the same (up to permutation).
For very large systems, the task of determining graph isomorphism of the total amalgamations can of course be formidable,
but often times in practice this is very user-friendly.

Boyle--Fiebig--Fiebig also provide an example to show that having conjugate higher powers does not coincide with one-sided conjugacy.
Below we generalise this example slightly and show that the one-sided systems determined from the matrices are Matsumoto eventually conjugate.

\begin{example}\label{ex:BFF}
  For each $k\in \N_+$, we consider the matrix
  \[
    A_k = 
    \begin{pmatrix}
      2k & 0  & 4k \\
      k  & 2k & 0  \\
      k  & 2k & 0  \\
    \end{pmatrix}.
  \]
  For $k=1$, this is~\cite[Example 5.6]{Boyle-Fiebig-Fiebig}.
  The total amalgamation of $A_k^2$ is the matrix $(4k)^2$
  while the total amalgamation of $A_k^3$ is the matrix $(4k)^3$.
  Therefore, it follows from~\cref{thm:BFF} that $A_k$ and $(4k)$ have conjugate higher powers. 
  On the other hand, $A_k$ is already in the form of a total amalgamation.
  This means that $A_k$ and $(4k)$ are not one-sided conjugate for any $k$.

  Next we observe that the matrices $A_k$ and $(4k)$ are in fact balanced strong shift equivalent via the matrix
  \[
    C = 
    \begin{pmatrix}
      2k & 4k \\
      k & 2k
    \end{pmatrix}.
  \]
  Explicitly, we consider the matrices
  \[
    S = 
    \begin{pmatrix}
      1 & 0 \\
      0 & 1 \\
      0 & 1
    \end{pmatrix}, \quad
    R_1 = 
    \begin{pmatrix}
      2k & 0  & 4k \\
      k  & 2k & 0
    \end{pmatrix},\quad \textrm{and} \quad
    R_2 =
    \begin{pmatrix}
      2k & 2k & 2k \\
      k  & k  & k
    \end{pmatrix}.
  \]
  Direct computations show that
  \[
    A_k = S R_1, \qquad S R_2 = 
    \begin{pmatrix}
      2k & 2k & 2k \\
      k  & k  & k  \\
      k  & k  & k
    \end{pmatrix}, \quad \textrm{and} \quad
    R_1 S = C = R_2 S.
  \]
  Since the total amalgamation of $S R_2$ is the matrix $(4k)$,
  the matrices $S, R_1, R_2$ implement a concrete balanced strong shift equivalence between $A_k$ and $(4k)$ for each $k\in \N$,
  so the one-sided shifts are Matsumoto eventually conjugate.
\end{example}

The above observations lead us to ask whether one-sided eventual conjugacy in the sense of having conjugate higher powers 
and in Matsumoto's sense simply coincide. 
This is not the case.

\begin{example}
  The adjacency matrices of the graphs in~\cref{ex:Brix-Carlsen2020a} are
  \[
    A = 
    \begin{pmatrix}
      0 & 2 & 2 \\
      1 & 0 & 0 \\
      1 & 0 & 0 
    \end{pmatrix}\quad \textrm{and} \quad
    B = 
    \begin{pmatrix}
      0 & 3 & 1 \\
      1 & 0 & 0 \\
      1 & 0 & 0 
    \end{pmatrix}.
  \]
  The total amalgamation of any power of $A$ is always a $2\times 2$-matrix,
  while the total amalgamation of any power of $B$ is a $3\times 3$-matrix.
  Therefore, the one-sided systems do not have conjugate higher powers. 
\end{example}

The examples so far have shown us that both having conjugate higher powers and being balanced strong shift equivalent
are nontrivial in the sense that they do not coincide with one-sided conjugacy. 
We have also seen that the two notions do not agree, in particular balanced strong shift equivalence does not imply having conjugate higher powers.
This is why we do not simply refer to the relations as \emph{one-sided eventual conjugacy}.
We leave the following problem open.

\begin{question}
  Does having conjugate higher powers imply balanced strong shift equivalence?
\end{question}

\begin{example}\label{ex:Ashley}
Consider Ashley's graph
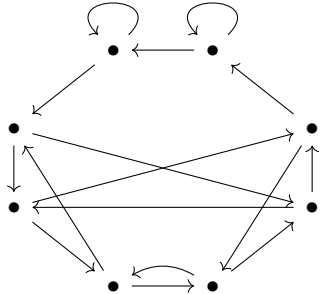
\[
  \begin{tikzcd}
    & \bullet \arrow[dl] \arrow[loop, looseness=5] & \bullet \arrow[l] \arrow[loop, looseness=5]&           \\
    \bullet \arrow[d] \arrow[rrrd] &         &         & \bullet \arrow[ul]  \arrow[ddl] \\
    \bullet \arrow[dr] \arrow[rrru] &         &         & \bullet \arrow[u] \arrow[lll] \\
    & \bullet \arrow[r] \arrow[uul] & \bullet \arrow[ur] \arrow[l, bend right] &           
  \end{tikzcd}
\]
The characteristic polynomial of the adjacency matrix of this graph is $x^7(x-2)$ and a result of Williams (see e.g.~\cite[Lemma 2.2.6]{Kitchens}) 
implies that the two-sided edge shift is shift equivalent to the matrix $(2)$, 
the full $2$-shift.
According to~\cite[Example 2.2.7 and the notes to that section]{Kitchens} it has been an open problem since 1989 when the graph was discovered 
to determine whether they are two-sided conjugate.

A computation of the higher powers of the adjacency matrix of Ashley's graph shows that all the higher powers are one-sided conjugate to the higher powers of the full $2$-shift (using~\cref{thm:BFF}),
so they have conjugate higher powers.
This makes the above question of whether this implies (balanced) strong shift equivalence even more pressing.
\end{example}

\begin{example}
  Consider Rourke's example of shift equivalent primitive matrices
  \[
    A = 
    \begin{pmatrix}
      1 & 2 & 1 \\
      1 & 1 & 0 \\
      1 & 0 & 1
    \end{pmatrix} \qquad \textrm{and} \qquad
    B = 
    \begin{pmatrix}
      1 & 0 & 1 & 0 & 1 \\
      0 & 1 & 1 & 1 & 0 \\
      1 & 1 & 1 & 0 & 0 \\
      1 & 0 & 0 & 0 & 1 \\
      0 & 1 & 0 & 1 & 0 
    \end{pmatrix}
  \]
  from the errata of~\cite{Williams1973}.
  Williams remarks that it is not known if they are strong shift equivalent.
  A computation of the higher powers shows that the total amalgamation of $B^n$ is $A^n$ for $n\geq 2$,
  so the matrices have conjugate higher powers.
  I do not know if the question of strong shift equivalence has been resolved, or if they are balanced strong shift equivalent.
\end{example}

\section{Flow equivalence}\label{sec:flow-equivalence}
In this section we return to two-sided dynamics.
Let us first define the notion of flow equivalence as in e.g.~\cite[Section 13.6]{Lind-Marcus}.

Let $X$ be a compact metric space equipped with a homeomorphism $\sigma\colon X\to X$.
The product $X\times \R$ carries an obvious $\R$-action (a flow)
and the \emph{suspension space} is the quotient of $X\times \R$ by the identification $(\sigma(x), t) \sim (x, t+1)$ for all $x\in X$ and $t\in \R$.
Let $[x,t] \in X\times \R/ \sim$ be the class of $(x,t)\in X\times \R$.
The quotient is a compact metric space and it carries an induced flow (the suspension flow) $s.[x,t] = [x,t+s]$ for all $x\in X$ and $t,s\in \R$.
A pair of homeomorphisms $\sigma_X\colon X\to X$ and $\sigma_Y\colon Y\to Y$ are \emph{flow equivalent} 
if their suspension flows are topologically equivalent,
i.e. there is a homeomorphism $h\colon X\times \R/ \sim \to X\times \R/ \sim$ that takes orbits (flow lines) to orbits in an orientation-preserving way.
Note that if the space $X$ is totally disconnected, then the flow lines are exactly the connected components of $X\times \R/\sim$,
so flow equivalence is a matter of 'oriented homeomorphism' on the suspension spaces.

The system $\sigma_X\colon X\to X$ appears as a cross section of the flow on the suspension space. 
Equivalently, two homeomorphisms are flow equivalent if they are cross sections of a common flow. 
As Franks mentions in~\cite{Franks1984}, symbolic dynamics is useful as a tool for studying the qualitative behaviour of certain systems of smooth dynamics.
By restricting to a transversal with its first return map (a cross section), often times the resulting homeomorphism is an irreducible shift of finite type. 
Restricting to a different transversal will result in a different system but they are flow equivalent.

In the beginning of the 1970s about the time when Williams was working on conjugacy, Bowen~\cite{Bowen1972} studied the flows coming from shifts of finite type.
The hugely influential (and notoriously short) paper by Parry and Sullivan~\cite{Parry-Sullivan} followed a few years later
in which the authors prove that flow equivalence of shifts of finite type is generated by conjugacy and \emph{symbol expansion}
(the reader \emph{with a less direct pipeline to topological truth} may want to consult e.g.~\cite{Boyle-Carlsen-Eilers2017}).
We illustrate symbol expansion with the following example that we saw in~\cref{ex:SFT}.

\begin{example}\label{ex:symbol-expansion}
  Consider the graphs $E$ and $F$ below
  \[
  \begin{tikzcd}
    \bullet \arrow[loop, looseness=5, "e" above] \arrow[loop, looseness=5,, out=310, in=230, "f"] 
  \end{tikzcd} \qquad \textrm{and}
  \begin{tikzcd}
    \bullet \arrow[loop, looseness=5, "e_1" above] \arrow[r, bend right, "f_1" below] & \bullet \arrow[l, bend right, "f_2" above]
  \end{tikzcd}
  \]
  We construct $F$ from $E$ by performing a symbol expansion on the symbol $f$ in $E$.
  Choose a new symbol, in this case $f_2$, and demand that it always follows $f$.
  In this way, we have slowed down the time it takes to walk along $f$.
  By identifying $f$ in $E$ with $f_1$ in $F$ we obtain the graph $F$.
  The reader may chack that the association $e \mapsto e_1$ and $f \mapsto f_1 f_2$ induces a flow equivalence between $E$ and $F$.

  Note that the edge shift of $E$ has no obstructions on its allowed paths (the full $2$-shift), so its entropy is $\log 2$,
  while it can be shown that the entropy of $F$ is $(1+\sqrt{5})/2$ (the golden mean), cf.~\cite[Example 4.1.4]{Lind-Marcus}.
  In particular, the systems are not conjugate and entropy is not invariant under flow equivalence.
\end{example}

Parry and Sullivan also showed that the integer $\det(\Id - A)$ is an invariant of the flow equivalence class of $(\sigma_A, \underline{X}_A)$
(in~\cref{ex:symbol-expansion} this value if $-1$ for both graphs).
Shortly after, Bowen and Franks~\cite{Bowen-Franks} showed that the kernel and cokernel of $\Id - A$ viewed as a linear map on $\Z^{|A|}$
(cf.~\labelcref{eq:Bowen-Franks}) are also invariant under flow equivalence 

The cokernel $\BF(A) = \Z^{|A|}/ (\Id-A)\Z^{|A|}$ (which today is usually referred to as the Bowen--Franks group) almost subsumes the information of the determinant since
\[
|\det(\Id-A)| =
  \begin{cases}
   |\BF(A)| & \textrm{if}~\BF(A)~\textrm{is finite}, \\
   0 & \textrm{if}~\BF(A)~\textrm{is infinite},
  \end{cases}
\]
so only the \emph{sign} of the determinant contains new information not retrievable from the group, cf.~\cref{ex:Cuntz-splice}.
We may therefore say that the \emph{signed Bowen--Franks group} 
\[
  (\BF(A), \sign \det(\Id-A))
\]
is an invariant of flow equivalent.
It is a surprising fact that for irreducible and nonpermutation matrices this is a \emph{complete} invariant.
This is Franks' classification result~\cite{Franks1984}.

\begin{theorem} 
  Let $A$ and $B$ be irreducible and nonpermutation $\N$-matrices.
The two-sided shifts $(\sigma_A, \underline{X}_A)$ and $(\sigma_B, \underline{X}_B)$ are flow equivalent if and only if $\BF(A) \cong \BF(B)$ and $\sign \det(\Id-A)) = \sign \det(\Id-A)$.
\end{theorem}

Using a different approach, Boyle and Handelman~\cite{Boyle-Handelman1996} studied the \emph{ordered cohomology} of a shift of finite type $(\sigma_A, \underline{X}_A)$ given as
\[
  C(\underline{X}_A, \Z) / \{ f-f\circ \sigma_A : f\in C(\underline{X}_A, \Z) \}
\]
with the order $[f]\geq 0$ if there exists a nonnegative representative $f\geq 0$.
For Cantor minimal systems, this is the group that Giordano, Putnam, and Skau used to classify Cantor minimal systems up to orbit equivalence, cf.~\cref{sec:CMS}.
Boyle and Handelman initiated a more general study of this ordered group, and proved that it classifies irreducible shifts of finite type up to flow equivalence~\cite[Theorem 1.12]{Boyle-Handelman1996}.
It might be the case that the ordered cohomology with an order unit is relevant for further classification of shifts of finite type, cf.~\cite[Problem 23.4]{Boyle2008}.

A complete invariant of flow equivalence for general shifts of finite type exists, but it is \emph{much} more complicated and involves the $K$-web.
The work of Boyle~\cite{Boyle2002} and Boyle--Huang~\cite{Boyle-Huang2003} develops this in a way that emphasises certain positivity and algebraic aspects separately.
For two-component shifts of finite type, Huang~\cite{Huang1994} has a satisfying classification of flow equivalence
using Frank's classification on the irreducible components as well as a distinguished element in a tensor product of Bowen--Franks groups,
the latter of which he refers to as the \emph{Cuntz group} since Cuntz found this to be an invariant of flow equivalence in~\cite[Section 4]{Cuntz1981b}.

We mention also here that that flow equivalence was proven to be decidable by Boyle and Steinberg~\cite[Corollary 2.4]{Boyle-Steinberg} 
although their proof only provides the existence of a decision algorithm.

Since the inception of the Cuntz--Krieger algebras, flow equivalence has played an important role.
Cuntz and Krieger showed~\cite[Theorem 4.1]{Cuntz-Krieger1980} that the stabilised C*-algebra $\OO_A\otimes \K$ together with its diagonal subalgebra $C(X_A)\otimes c_0$ 
is an invariant of the flow class of $(\sigma_A, \underline{X}_A)$.
They point out that using Parry and Sullivan's characterisation (the symbol expansion), this is an \emph{instant computational proof}.
The following example should be seen in this light.

\begin{example}\label{ex:Cuntz-splice}
  Consider the graphs
  
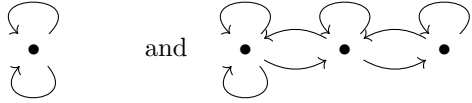
\[
    \begin{tikzcd}
      \bullet \arrow[loop, looseness=5,] \arrow[loop, looseness=5,, out=310, in=230,] 
    \end{tikzcd} \qquad \textrm{and}
    \begin{tikzcd}
      \bullet \arrow[loop, looseness=5] \arrow[loop, looseness=5,, out=310, in=230] \arrow[r, bend right] &
      \bullet \arrow[loop, looseness=5] \arrow[l, bend right] \arrow[r, bend right] &
      \bullet \arrow[loop, looseness=5] \arrow[l, bend right]
    \end{tikzcd}
  \]
  and note that the Bowen--Franks groups are trivial for both graphs.
  The two-sided shifts are however not flow equivalent as the reader may check by computing the sign of the determinant.

  The C*-algebra of the graph with a single vertex and two loops is Cuntz' algebra $\OO_2$, 
  and the C*-algebra of the right-most graph is usually known as $\OO_{2-}$.
  For a long time it was not known whether the two were stably isomorphic,
  and it was only with R\o rdam's clasification of simple Cuntz--Krieger algebras it was finally concluded that $\OO_2 \cong \OO_{2-}$.

  For the cognoscenti, this is an example of a \emph{Cuntz splice}; a move designed to flip the sign of the determinant and thereby change the flow class of the graph.
\end{example}

The fact that two systems that are dynamically far apart (they are not flow equivalent)
can have *-isomorphic C*-algebras was a mystery.
How much is actually remembered (or lost) in the transition from dynamical systems or graphs to C*-algebras?

The mystery was resolved only much later by Matsumoto and Matui.
Their C*-algebraic characterisation of flow equivalence for infinite irreducible shifts of finite type is from~\cite[Corollary 3.8]{Matsumoto-Matui2014}.
The result was later generalised to include the reducible shifts of finite type in~\cite{CEOR}.

\begin{theorem} \label{thm:MM14-flow}
  Let $A$ and $B$ be square $\N$-matrices.
  The two-sided shifts $(\sigma_A, \underline{X}_A)$ and $(\sigma_B, \underline{X}_B)$ are flow equivalent if and only there exists a stable *-isomorphism
  $\varphi\colon \OO_A\otimes \K \to \OO_B\otimes \K$ satisfying $\varphi(C(X_A)\otimes c_0) = C(X_B)\otimes c_0$.
\end{theorem}

This means that although $\OO_2 \cong \OO_{2-}$, cf.~\cref{ex:Cuntz-splice}, this *-isomorphism cannot be diagonal-preserving.

As mentioned above, one direction was already observed by Cuntz and Krieger.
The proof of the other direction (assuming stable diagonal *-isomorphism and showing flow equivalence) 
is remarkable in that it used Matusmoto's notion of continuous orbit equivalence (see~\cref{sec:continuous-orbit-equivalence}), its connection to C*-algebras,
and the fact that it implies flow equivalence, \'etale topological groupoids and their cohomology which the authors identify with the ordered cohomology of Boyle and Handelman,
Huang's theorem~\cite[Theorem 2.15]{Huang1994} that any automorphism of the Bowen--Franks group is induced by a flow equivalence,
and finally Franks' classification in terms of the signed Bowen--Franks invariant.

This proof strategy does not quite work for general shifts of finite type,
and the interesting aspect of the generalisation to the reducible case
is the construction of an explicit flow equivalence~\cite[Section 3.1]{CEOR}.
It remains an open problem if a similar result holds for general subshifts
although some progress was made in~\cite[Section 8]{Brix-Carlsen2020b}.

\begin{question}
  Let $(\sigma_X, \underline{X})$ and $(\sigma_Y, \underline{Y})$ be general subshifts and suppose there exists a *-isomorphism 
$\varphi\colon \OO_X\otimes \K \to \OO_Y\otimes \K$ satisfying $\varphi(C(X)\otimes c_0) = C(Y)\otimes c_0$.
Does it follow that $(\sigma_X, \underline{X})$ and $(\sigma_Y, \underline{Y})$ are flow equivalent?
\end{question}

Eilers, Restorff, and Ruiz~\cite[Example 4.4]{Eilers-Restorff-Ruiz2009} exhibit an example of two substitution shifts that are not flow equivalent
but whose C*-algebras are stably isomorphic.

\subsection{Back to shift equivalence}
It is easy to verify that conjugacy implies flow equivalence (from the definitions).
Let us briefly here mention a connection to shift equivalence, cf.~\cref{sec:shift-equivalence}.
The Bowen--Franks group is invariant under shift equivalence, and so is the spectrum away from zero,
so $\det(\Id - A)$ is also preserved. 
By Franks' classification we therefore have the following observation.

\begin{corollary} \label{cor:SEimpliesFE}
  For irreducible and nonpermutation $\N$-matrices,
  shift equivalence implies flow equivalence.
\end{corollary}

Interestingly, Huang~\cite[Corollary 3.4]{Huang1994} uses the classification of flow equivalence of two-component shifts of finite type
to show that shift equivalence implies flow equivalence also in this case; 
Huang also shows~\cite[Corollary 3.12]{Huang1995} that this happens whenever the Bowen--Franks group is finite
(I thank Efren Ruiz for alerting me to this result).
It is still not known whether this happens for all reducible shifts of finite type.

A somewhat surprising way of rephrasing~\cref{cor:SEimpliesFE} is by combining Bratteli and Kishimoto's characterisation of shift equivalence~\cref{thm:Bratteli-Kishimoto}
with Matsumoto and Matui's result on flow equivalence~\cref{thm:MM14-flow}.

\begin{corollary} 
  Let $A$ and $B$ be primitive $\N$-matrices.
  If there is a *-isomorphism $\varphi\colon \OO_A\otimes \K \to \OO_B\otimes \K$ satisfying $\varphi\circ (\gamma^A\otimes \Id) = (\gamma^B\otimes \Id)\circ \varphi$,
  then there is a (possibly different) *-isomorphism $\varphi' \colon \OO_A\otimes \K \to \OO_B\otimes \K$ satisfying $\varphi'(C(X_A)\otimes c_0) = C(X_B)\otimes c_0$.
\end{corollary}

This means that the *-isomorphism in~\cref{ex:Cuntz-splice} cannot be equivariant,
since this would imply flow equivalence.

It is important to note that we cannot expect $\varphi$ and $\varphi'$ to be identical;
that is, we cannot expect an equivariant *-isomorphism to automatically also be diagonal-preserving.
If this were the case, then we would have a two-sided conjugacy, cf.~\cref{thm:Carlsen-Rout},
but we know from the counterexample of Kim and Roush~\cref{ex:Kim-Roush} that this is not necessarily the case.

\section{Continuous orbit equivalence} \label{sec:continuous-orbit-equivalence}
Finally, we address the notion of one-sided continuous orbit equivalence.
This should be thought of as the one-sided analogue of flow equivalence which we just discussed.
We should emphasise that the one-sided case is quite distinct from the two-sided case.

\subsection{Cantor minimal systems} \label{sec:CMS}
For a minimal homeomorphism $\phi\colon X\to X$ on the Cantor space $X$ (short: \emph{Cantor minimal systems}), 
Giordano, Putnam, and Skau~\cite{Giordano-Putnam-Skau1995} managed to completely classify $\phi$ up to (variations of) orbit equivalence.
Important examples of Cantor minimal systems include odometers and Sturmian shifts.
The crossed product C*-algebra $C(X)\rtimes_\phi \Z$ is a simple A$\T$-algebra (inductive limit of circle algebras)
and its $K_0$-group is a simple dimension group that classifies these crossed products up to *-isomorphism 
when considered as an ordered group with a distinguished order unit (the $K_1$-group is always isomorphic to the integers).
The $K_0$-group is order isomorphic to (and may hence be described concretely as) the coinvariants
\begin{equation} \label{eq:co-invariants}
  K^0(X, \phi) = C(X, \Z) / \{ f - f\circ \phi^{-1} : f\in C(X, \Z) \}
\end{equation}
with a distinguished order unit given by the class of the function $1$.
An element $a\in K^0(X,\phi)$ is infinitesimal if $-\varepsilon 1 \leq a \leq \varepsilon 1$ for all rational $0< \varepsilon$,
and the infinitesimals form a subgroup $\Inf(K^0(X,\phi))$ of $K^0(X,\phi)$.

An orbit equivalence is a homeomorphism $h\colon X_1 \to X_2$ that sends the orbit of any $x\in X_1$ onto the orbit of $h(x)\in X_2$.
Equivalently, there are cocycle functions $n\colon X_1\to \Z$ and $m\colon X_2 \to \Z$
such that $h\circ \phi_1(x) = \phi_2^{n(x)}\circ h(x)$ for all $x\in X_1$ and $h^{-1}\circ \phi_2(y) = \phi_1^{m(y)}\circ h^{-1}(y)$ for all $y\in X_2$.
Strong orbit equivalence assumes that the cocycle functions have at most one point of discontinuity,
and continuous orbit equivalence simply assumes that the cocycle functions are continuous.

The \emph{full group} $[\phi]$ of $(X,\phi)$ is the group consisting of all self-orbit equivalences, 
i.e. a homeomorphism $\psi$ is in $[\phi]$ if there is a function $n\colon X \to \Z$ such that $\psi(x) = \phi^{n(x)}(x)$ for all $x\in X$.
The \emph{topological full group} $[[\phi]]$ is the subgroup of orbit equivalences for which the cocycle function $n$ can be taken to be continuous.
Giordano, Putnam, and Skau proved that Cantor minimal systems $(X_1,\phi_1)$ and $(X_2,\phi_2)$ are 
\begin{itemize}
  \item topologically orbit equivalent precisely when their dimension groups with distinguished unit modulo the infinitesimals are isomorphic~\cite[Theorem 2.2]{Giordano-Putnam-Skau1995},
    and if and only the full groups are isomorphic~\cite[Corollary 4.6]{Giordano-Putnam-Skau1999};
  \item strong orbit equivalent precisely when their dimension groups with units are isomorphic, 
    and that this happens if and only if the crossed products are *-isomorphic~\cite[Theorem 2.2]{Giordano-Putnam-Skau1995}; and
  \item continuously orbit equivalent precisely when their crossed products are *-isomorphic in a diagonal-preserving way~\cite[Theorem 2.4]{Giordano-Putnam-Skau1995},
    and if and only if their topological full groups are isomorphic~\cite[Corollary 4.4]{Giordano-Putnam-Skau1999}.
\end{itemize}
In the third bullet above, we emphasise the definition of continuous orbit equivalence because it provides a nice analogue to the one-sided case
that we shall discuss shortly.
However, Boyle~\cite{Boyle1983} showed in his thesis that this coincides with flip conjugacy (i.e. that $\phi_1$ is conjugate to either $\phi_2$ or $\phi_2^{-1}$),
and the facts that continuous orbit equivalence implies flip conjugacy and 
that this coincides with diagonal-preserving *-isomorphism of crossed products was later generalised to the setting of topologically transitive and topologically free homeomorphisms of compact Hausdorff spaces
by Boyle and Tomiyama~\cite{Boyle-Tomiyama1998}.

The topological full groups from Cantor minimal systems admit interesting group theoretic properties.
Indeed, Jushenko and Monod~\cite[Theorem A]{Juschenko-Monod2013} showed that these groups are always amenable,
and by restricting to those Cantor minimal systems that are conjugate to minimal shifts, they exhibited the first examples of (actually uncountably many) infinite finitely generated simple and amenable groups,
cf.~\cite{Matui2006}.

If we consider homeomorphisms on compact Hausdorff spaces as $\Z$-actions then flip conjugacy is really the correct notion of conjugacy of such actions,
so the Boyle--Tomiyama result says that topologically transitive and topologically free $\Z$-actions on compact Hausdorff spaces are continuously orbit equivalence rigid.
Continuous orbit equivalence for general group actions and its connection to diagonal-preserving *-isomorphism of crossed products 
was thoroughly studied by Li in~\cite{Li2018} in which examples that are not continuously orbit equivalence rigid also appear.

\subsection{One-sided continuous orbit equivalence}
For irreducible shifts of finite type, the notion of one-sided continuous orbit equivalence was introduced by Matsumoto~\cite{Matsumoto2010}.
The general definition can be found e.g.~in~\cite[Definition 8.1]{CRST}.
\begin{definition} \label{def:coe}
  Let $T\colon X\to X$ and $S\colon Y\to Y$ be local homeomorphisms on locally compact Hausdorff spaces $X$ and $Y$, respectively.
  Then $T$ and $S$ are \emph{continuously orbit equivalent} if there exist a homeomorphism $h\colon X\to Y$
  and continuous functions $k_X,l_X\colon X\to \N$ and $k_Y,l_Y\colon Y\to \N$ such that
  \begin{align} 
    S^{l_X(x)}\circ h(x) &= S^{k_X(x)}\circ h\circ T(x), \label{eq:coe1} \\
    T^{l_Y(y)}\circ h^{-1}(y) &= T^{k_Y(y)}\circ h^{-1}\circ S(y), \label{eq:coe2} 
  \end{align}
  for all $x\in X$ and $y\in Y$.
\end{definition}

\begin{remark}
  Note that if in the above definition, the cocycles for $h$ can be chosen to be $l_X = k_X + 1$ and $l_Y = k_Y + 1$,
  then $h$ is a Matsumoto eventual conjugacy, cf.~\cref{def:Matsumoto-eventual-conjugacy}.
\end{remark}

\begin{example}
  The graphs in~\cref{ex:symbol-expansion} are continuously orbit equivalent;
  an explicit map can be constructed using the same symbol expansion as in that example, see~\cite[p. 203]{Matsumoto2010}.
\end{example}

Matsumoto had a clear motivation from C*-algebras, and~\cite[Theorem 1.1]{Matsumoto2010} shows that continuous orbit equivalence
is the same as diagonal-preserving *-isomorphism of Cuntz--Krieger algebras.
Later in~\cite[Theorem 3.6]{Matsumoto-Matui2014}, Matsumoto and Matui characterised continuous orbit equivalence in terms of the \emph{unital} signed Bowen--Franks group.
The \emph{unital} Bowen--Franks group is the cokernel $\BF(A) = \Z^{|A|}/ (\Id-A)\Z^{|A|}$ together with the class of the vector $(1,\ldots,1)\in \Z^{|A|}$ in $\BF(A)$,
we denote the class as $u_A$.
We summarise these results.

\begin{theorem} \label{thm:coe-unital-signed-BF}
  Let $A$ and $B$ be irreducible and nonpermutation $\N$-matrices.
  The one-sided shifts of finite type $(\sigma_A, X_A)$ and $(\sigma_B, X_B)$ are continuously orbit equivalent 
  if and only if $(\BF(A), u_A) \cong (\BF(B), u_B)$ and $\det(\Id - A) = \det(\Id - B)$,
  and if and only if there is a *-isomorphism $\varphi\colon \OO_A \to \OO_B$ satisfying $\varphi(C(X_A)) = C(X_B)$.
\end{theorem}

Part of the proof involves showing that continuous orbit equivalence implies flow equivalence (see~\cref{sec:flow-equivalence}).
The connection between continuous orbit equivalence and diagonal-preserving *-isomorphism of Cuntz--Krieger algebras
and the fact that continuous orbit equivalence implies flow equivalence
was again generalised to reducible shifts of finite type in~\cite{CEOR}.

\begin{example}
  Ashley's graphs~\cref{ex:Ashley} are continuously orbit equivalent:
  since the two systems are primitive and flow equivalent, and the Bowen--Franks group are trivial (so the condition on the class of the unit is trivially satisfied),
  this follows from the above theorem.
\end{example}

\begin{example}
  Kim and Roush's primitive counterexample to the Williams problem~\cref{ex:Kim-Roush}
  are continuously orbit equivalent, cf.~\cite[Example 4.9]{Eilers-Ruiz}.
\end{example}

\begin{example}
  The matrices
  \[
    A = 
    \begin{pmatrix}
      1 & 1 & 1 \\
      1 & 1 & 1 \\
      1 & 0 & 0
    \end{pmatrix}, \qquad \textrm{and} \qquad
    B =
    \begin{pmatrix}
      1 & 1 & 1 \\
      1 & 1 & 0 \\
      1 & 1 & 0
    \end{pmatrix}.
  \]
  are strong shift equivalent:
  $A$ is an out-split and $B$ is an in-split of $ \begin{pmatrix}  1 & 2 \\  1 & 1 \end{pmatrix}$.
  However, $\OO_A \cong \OO_3$ and $\OO_B \cong \OO_3\otimes M_2$, so they are not continuously orbit equivalent, cf.~\cite[Section 3]{Matsumoto2017-ucoe}.
\end{example}

An algebraic characterisation of continuous orbit equivalence for reducible shifts of finite type that generalises the unital signed Bowen--Franks invariant seems to still be missing.
Therefore, it is still not clear whether this is decidable.

For local homeomorphisms,~\cite[Theorem 8.2]{CRST} shows that continuous orbit equivalence is characterised by diagonal-preserving *-isomorphism.
The proof of this relies on groupoid reconstruction, cf.~\cref{thm:Renault-reconstruction}.

\begin{theorem} \label{thm:CRST-coe}
  Let $T\colon X \to X$ and $S\colon Y \to Y$ be local homeomorphisms.
  Then $T$ and $S$ are (stabiliser-preserving) continuously orbit equivalent if and only if there is a *-isomorphism
  $\varphi\colon C^*(T) \to C^*(S)$ satisfying $\varphi(C_0(X)) = C_0(Y)$.
\end{theorem}

This is used to show in~\cite[Theorem 6.10]{Brix-Carlsen2020b} (see also~\cite[Theorem 1.2]{Matsumoto2020}) that (under certain technical conditions)
shift spaces $(\sigma_X, X)$ and $(\sigma_Y,Y)$ are continuously orbit equivalent
if and only if there is a *-isomorphism $\varphi\colon \OO_X \to \OO_Y$ satisfying $\varphi(C(X)) = C(Y)$.
Note that this is not immediate from~\cref{thm:CRST-coe}:
the shifts $\sigma_X$ and $\sigma_Y$ are generally \emph{not} local homeomorphisms (only when they are of finite type)
but the actions on the covers $(\sigma_{\tilde{X}}, \tilde{X})$ and $(\sigma_{\tilde{Y}}, \tilde{Y})$ are,
so the theorem can only be applied to the covers;
the difficulty is then to understand how a continuous orbit equivalence among shift spaces passes to the covers (and back again). 
At the moment, it is not clear if the technical conditions can be removed. 

We also state a question related to flow equivalence.

\begin{question}
  Does continuous orbit equivalence imply flow equivalence?
\end{question}

\begin{example}
The labelled graphs
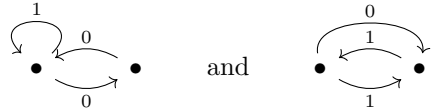
\[
  \begin{tikzcd}
    \bullet \arrow[loop, looseness=5, "1" above] \arrow[r, bend right, "0" below] & \bullet \arrow[l, bend right, "0" above]
  \end{tikzcd} \qquad \textrm{and} \qquad
  \begin{tikzcd}
    \bullet \arrow[r, bend left=100, "0" above] \arrow[r, bend right, "1" below] & \bullet \arrow[l, bend right, "1" above]
  \end{tikzcd} 
\]
represent the even and odd shifts, respectively.
We encountered the left-most graph as representing a sofic shift in~\cref{ex:even-shift}.
In the right-most labelled path space, only an odd number of $0$s are allowed between any two $1$s,
so this is again a sofic shifts.
A continuous orbit equivalence between the even and odd shifts is given by sending any occurrence of $1$ to $10$, cf.~\cite[Example 6.15]{Brix-Carlsen2020b}
(a similar argument shows that they are flow equivalent).
\end{example}

Following the ideas of Giordano, Putnam, and Skau, cf.~\cref{sec:CMS}, it is also possible to characterise the one-sided shifts of finite type up to continuous orbit equivalence 
using a notion of topological full group.
Let $[[\sigma_A]]$ be the subgroup of homeomorphisms $h$ of $X_A$ for which there exist continuous cocycles $k_X,l_X\colon X_A \to \N$ such that relations similar to those in~\cref{def:coe} are satisfied.
Infinite and irreducible shifts of finite type $X_A$ and $X_B$ are then continuously orbit equivalent precisely if there exists a homeomorphism $h\colon X_A \to X_B$ such that 
$h\circ [[\sigma_A]] \circ h^{-1} = [[\sigma_B]]$, cf.~\cite[Theorem 1.1]{Matsumoto2010}. 

For the full $2$-shift, the topological full group $[[\sigma_{(2)}]]$ is isomorphic to Thompson's group $V$, 
and $[[\sigma_{(n)}]]$ is isomorphic to the Higman--Thompson group $V_n$, cf. \cite[Remark 6.3]{Matui2015}.
Matui comments that from this perspective, the topological full groups of shifts of finite type may be viewed as generalisations of Higman--Thompson groups.
 
In fact, Matui vastly generalised the notion of topological full group to the setting of \'etale groupoids~\cite{Matui2012}.
Suppose $G$ is a locally compact Hausdorff and \'etale groupoid whose unit space $G^{(0)}$ is the Cantor space.
The compact open bisections (the subsets of $G$ on which the range and source maps are injective) form an inverse semigroup,
and the topological full group $[[G]]$ of $G$ is the subgroup consisting of the global bisections (those for which the range and source is the whole unit space).

Matui proves that the (abstract) isomorphism class of the topological full groups characterise the isomorphism class of the groupoids, cf.~\cite[Theorem 3.10]{Matui2015}.
Again we refer the reader to~\cite{Sims-notes} for the basics of \'etale groupoids.
The groupoid $G_A$ of a one-sided shift of finite type $X_A$ is minimal and effective precisely when $A$ is an irreducible and nonpermutation matrix.

\begin{theorem} 
  Let $G_1$ and $G_2$ be second-countable, minimal, effective, and \'etale groupoids whose unit spaces are Cantor spaces.
  Then $G_1$ and $G_2$ are isomorphic as topological groupoids if and only if the topological full groups $[[G_1]]$ and $[[G_2]]$ are abstractly isomorphic as groups.
\end{theorem}

Comparing this to Renault's reconstruction theorem~\cref{thm:Renault-reconstruction} we see how this vastly generalises Matsumoto's observations.
The groups arising from continuous orbit equivalences of dynamical systems thus exhibit interesting group theoretic properties
as well as interesting connections to dynamics, groupoids, and C*-algebras.


\begin{thebibliography}{99}
\bibliographystyle{amsalpha}


\bibitem[ABCE22]{Armstrong-Brix-Carlsen-Eilers} B.~Armstrong, K.A.~Brix, T.M.~Carlsen, and S.~Eilers,
\emph{Conjugacy of local homeomorphisms via groupoids and C*-algebras},
Ergodic Theory Dynam. Systems (2022), 1--22, doi:10.1017/etds.2022.50.

\bibitem[ABCEW]{ABCEW} B.~Armstrong, K.A.~Brix, T.M.~Carlsen, S.~Eilers, and J.~Winkel,
  in preparation.

\bibitem[BP04]{Bates-Pask2004} T. Bates and D. Pask
\emph{Flow equivalence of graph algebras},
Ergodic Theory Dynam. Systems \textbf{24} (2004), 367--382.

\bibitem[Bo72]{Bowen1972} R.~Bowen,
\emph{One-dimensional hyperbolic sets for flows},
J. Differential Equations \textbf{12} (1972), 173--179.

\bibitem[Bo78]{Bowen1978} R.~Bowen,
\emph{On Axiom A diffeomorphisms},
Regional Conference Series in Mathematics, No. 35. American Mathematical Society, Providence, R.I., 1978. vii+45 pp. ISBN: 0--8218--1685--3.

\bibitem[BF77]{Bowen-Franks} R. Bowen and J. Franks,
\emph{Homology for zero-dimensional nonwandering sets},
Ann.~of Math. (2) \textbf{106} (1977), no. 1, 73--92.

\bibitem[Bo83]{Boyle1983} M.~Boyle,
\emph{Topological orbit equivalence and factor maps in symbolic dynamics},
Ph.D. thesis, University of Washington (1983).

\bibitem[Bo02]{Boyle2002} M.~Boyle,
\emph{Flow equivalence of shifts of finite type via positive factorizations},
Pacific J. Math. \textbf{204} (2002), no. 2, 273--317. 

\bibitem[Bo08]{Boyle2008} M. Boyle,
\emph{Open problems in symbolic dynamics},
Geometric and probabilistic structures in dynamics, 69--118,
Contemp. Math., \textbf{469}, Amer. Math. Soc., Providence, RI, 2008.

\bibitem[BCE17]{Boyle-Carlsen-Eilers2017} M. Boyle, T.M. Carlsen and S. Eilers,
\emph{Flow equivalence and isotopy for subshifts},
Dyn. Syst. \textbf{32} (2017), no. 3, 305--325.

\bibitem[BFF97]{Boyle-Fiebig-Fiebig} M.~Boyle, D.~Fiebig and U.~Fiebig,
\emph{A dimension group for local homeomorphisms and endomorphisms of onesided shifts of finite type}
J. reine angew. Math., \textbf{487} (1997), 27--59.

\bibitem[BH96]{Boyle-Handelman1996} M.~Boyle and D.~Handelman,
\emph{Orbit equivalence, flow equivalence and ordered cohomology},
Israel J. Math. \textbf{95} (1996), 169--210.

\bibitem[BH03]{Boyle-Huang2003} M.~Boyle and D.~Huang,
\emph{Poset block equivalence of integral matrices},
Trans. Amer. Math. Soc. \textbf{355} (2003), no. 10, 3861--3886. 

\bibitem[BK88]{Boyle-Krieger1988} M.~Boyle and W.~Krieger,
\emph{Almost Markov and shift equivalent sofic systems},
Proceedings of Maryland special year in dynamics 1986--1987,
Springer-Verlag lecture notes in math \textbf{1342} (1988), 33--93.

\bibitem[BS]{Boyle-Steinberg} M.~Boyle and B.~Steinberg,
\emph{Decidability of flow equivalence and isomorphism problems for graph C*-algebras and quiver representations}
to appear in Proceedings AMS, 12 pages, (arXiv:1812.04555 [math.OA]).

\bibitem[BT98]{Boyle-Tomiyama1998} M.~Boyle and J.~Tomiyama,
\emph{Bounded topological orbit equivalence and C*-algebras},
J. Math. Soc. Japan \textbf{50} (1998), no. 2, 317--329.

\bibitem[BK00]{Bratteli-Kishimoto} O. Bratteli and A. Kishimoto,
\emph{Trace scaling automorphisms of certain stable AF algebras. II},
Q. J. Math. \textbf{51} (2000), no. 2, 131--154.

\bibitem[Br22]{Brix2022} K.A.~Brix,
\emph{Balanced strong shift equivalence, balanced in-splits, and eventual conjugacy},
Ergodic Theory Dynam. Systems \textbf{42} (2022), no. 1, 19--39.

\bibitem[Br21]{Brix2021} K.A.~Brix,
\emph{Sturmian subshifts and their C*-algebras},
to appear in J. Operator Theory, (arXiv:2107.10613 [math.OA]).

\bibitem[Br]{Brix-unitalSE} K.A.~Brix,
\emph{Unital shift equivalence},
in preparation.

\bibitem[BC20a]{Brix-Carlsen2020a} K.A.~Brix and T.M.~Carlsen,
\emph{Cuntz--Krieger algebras and one-sided conjugacy of shift of finite type and their groupoids},
J. Aust. Math. Soc. \textbf{109} (2020), no. 3, 289--298.

\bibitem[BC20b]{Brix-Carlsen2020b} K.A. Brix and T.M. Carlsen,
\emph{C*-algebras, groupoids and covers of shift spaces},
Trans. Amer. Math. Soc. Ser. B \textbf{7} (2020), 134--185.

\bibitem[Ca03]{Carlsen2003} T.M.~Carlsen,
\emph{On C*-algebras associated with sofic shifts},
J. Operator Theory \textbf{49} (2003), no. 1, 203--212.

\bibitem[CEOR19]{CEOR} T.M. Carlsen, S. Eilers, E. Ortega and G. Restorff,
\emph{Flow equivalence and orbit equivalence for shifts of finite type and isomorphism of their groupoids},
J. Math. Anal. Appl. \textbf{469} (2019), 1088--1110.

\bibitem[CM04]{Carlsen-Matsumoto2004} T.M. Carlsen and K. Matsumoto,
\emph{Some remarks on the C*-algebras associated with subshifts},
Math. Scand. \textbf{95} (2004), 145--160.

\bibitem[CR17]{Carlsen-Rout2017} T.M. Carlsen and J. Rout,
\emph{Diagonal-preserving gauge-invariant isomorphisms of graph C*-algebras},
J. Funct. Anal. \textbf{273} (2017), 2981--2993.

\bibitem[CRST21]{CRST} T.M.~Carlsen, E.~Ruiz, A.~Sims and M.~Tomforde,
\emph{Reconstruction of groupoids and C*-rigidity of dynamical systems},
Adv. Math. \textbf{390} (2021).

\bibitem[Cu77]{Cuntz1977} J. Cuntz,
\emph{Simple C*-algebras generated by isometries},
Commun.  math.  Phys. \textbf{57} (1977), 173--185.

\bibitem[Cu81a]{Cuntz1981a} J.~Cuntz,
\emph{K-theory for certain C*-algebras},
Ann. of Math. (2) \textbf{113} (1981), no. 1, 181--197.

\bibitem[Cu81b]{Cuntz1981b} J. Cuntz,
\emph{A class of C*-algebras and topological Markov chains. {II}. Reducible chains and the Ext-functor for C*-algebras},
Invent. Math. \textbf{63} (1981), 25--40.

\bibitem[CK80]{Cuntz-Krieger1980} J.~Cuntz and W.~Krieger,
\emph{A class of C*-algebras and topological Markov chains}, 
Invent. Math. \textbf{56} (1980), 251--268.

\bibitem[ERR09]{Eilers-Restorff-Ruiz2009} S. Eilers, G. Restorff and E. Ruiz,
\emph{Classification of extensions of classifiable C*-algebras},
Adv. Math. \textbf{222} (2009), no. 6, 2153--2172.

\bibitem[ERRS16]{ERRS2016} S. Eilers, G. Restorff, E. Ruiz and A.P.W. S\o rensen,
\emph{The complete classification of unital graph C*-algebras: Geometric and strong},
Duke Math. J. \textbf{170} (2021), 2421--2517.

\bibitem[ER]{Eilers-Ruiz} S. Eilers and E. Ruiz,
\emph{Refined moves for structure-preserving isomorphism of graph C*-algebras},
arXiv:1908.03714v1 [math.OA], 46 pages.

\bibitem[EW80]{Enomoto-Watatani1980} M.~Enomoto and Y.~Watatani,
\emph{A graph theory for C*-algebras},
Math. Japon., \textbf{25}, no. 4, (1980), 435--442.

\bibitem[Fi75]{Fischer1975} R. Fischer,
\emph{Sofic systems and graphs},
Monatsh. Math. \textbf{80} (1975), no. 3, 179--186.

\bibitem[Fr84]{Franks1984} J.~Franks,
\emph{Flow equivalence of subshifts of finite type},
Ergodic Theory Dynam. Systems \textbf{4} (1984), no. 1, 53--66.

\bibitem[GPS95]{Giordano-Putnam-Skau1995} T. Giordano, I.F. Putnam and C.F. Skau,
\emph{Topological orbit equivalence and C*-crossed products},
J. reine angew. Math. \textbf{469} (1995), 51--111.

\bibitem[GPS99]{Giordano-Putnam-Skau1999}, T. Giordano, I.F. Putnam and C.F. Skau,
\emph{Full groups of Cantor minimal systems},
Israel J. Math. \textbf{111} (1999), 285--320.

\bibitem[HN88]{Hamachi-Nasu1988} T.~Hamachi and M.~Nasu,
\emph{Topological conjugacy for 1-block factor maps of subshifts and sofic covers},
Dynamical systems (College Park, MD, 1986–87), 251--260,
Lecture Notes in Math., 1342, Springer, Berlin, 1988.

\bibitem[Hu94]{Huang1994} L. Huang,
\emph{Flow equivalence of reducible shifts of finite type},
Ergodic Theory Dynam. Systems \textbf{14} (1994), no. 4, 695--720.

\bibitem[Hu95]{Huang1995} L. Huang,
\emph{Flow equivalence of reducible shifts of finite type and Cuntz--Krieger algebras}
J. reine angew. Math. \textbf{462} (1995), 185--217.

\bibitem[aHR97]{anHuef-Raeburn1997} A.~an Huef and I.~Raeburn,
\emph{The ideal structure of Cuntz--Krieger algebras},
Ergodic Theory Dynam. Systems \textbf{17} (1997), no. 3, 611--624.

\bibitem[JM13]{Juschenko-Monod2013} K. Juschenko and N. Monod,
\emph{Cantor systems, piecewise translations and simple amenable groups},
Ann.~of Math. (2) \textbf{178} (2013), 775--787.

\bibitem[Ka09]{Katsura2009} T. Katsura,
\emph{Cuntz-Krieger algebras and C*-algebras of topological graphs},
Acta Appl. Math. \textbf{108} (2009), no. 3, 617--624.

\bibitem[KR79]{Kim-Roush1979} K.H.~Kim and F.W.~Roush,
\emph{Some results on decidability of shift equivalence},
J. Comb. Inf. Syst. Sci. \textbf{4}(2) (1979), 123--146. 
  
\bibitem[KR88]{Kim-Roush1988} K.H.~Kim and F.W.~Roush,
\emph{Decidability of shift equivalence},
in Dynamical Systems, College Park, MD. 
Lecture Notes in Math., vol. 1342, pp. 374--424, Springer, Berlin (1988)

\bibitem[KR92]{Kim-Roush1992} K.H. Kim and F.W. Roush,
\emph{Williams's conjecture is false for reducible subshifts},
J. Amer. Math. Soc. \textbf{5} (1992), no. 1, 213--215.

\bibitem[KR99]{Kim-Roush1999} K.H. Kim and F.W.Roush,
\emph{The Williams conjecture is false for irreducible subshifts},
Ann.\ of Math. (2) \textbf{149} (1999), no. 2, 545--558.

\bibitem[Ki98]{Kitchens} B.P. Kitchens,
\emph{Symbolic dynamics}, One-sided, two-sided and countable state Markov shifts,
Universitext, Springer-Verlag, Berlin (1998), MR 1484730.

\bibitem[Ku86]{Kumjian1986} A. Kumjian,
\emph{On C*-diagonals},
Canad. J. Math. \textbf{38} (1986), no. 4, 969--1008.

\bibitem[Kr80a]{Krieger1980a} W.~Krieger,
\emph{On a dimension for a class of homeomorphism groups},
Math. Ann. \textbf{252} (1980), no. 2, 87--95.

\bibitem[Kr80b]{Krieger1980b} W.~Krieger,
\emph{Dimension groups functions and Topological Markov Chains},
Invent. Math. \textbf{56} (1980), no. 3, 239--250.

\bibitem[Kr84]{Krieger1984} W.~Krieger,
\emph{On sofic systems. I},
Israel J. Math. \textbf{48} (1984), no. 4, 305--330.

\bibitem[KPRR97]{Kumjian-Pask-Raeburn-Renault1997} A. Kumjian, D. Pask, I. Raeburn and J. Renault,
\emph{Graphs, groupoids, and {C}untz-{K}rieger algebras},
J. Funct. Anal. \textbf{144} (1997), 505--541.

\bibitem[Li18]{Li2018} X.~Li,
\emph{Continuous orbit equivalence rigidity},
Ergodic Theory Dynam. Systems \textbf{38} (2018), no. 4, 1543--1563.

\bibitem[LM95]{Lind-Marcus} D.~Lind and B.~Marcus,
\emph{An introduction to symbolic dynamics and coding}, Second edition,
Cambridge University Press, Cambridge, (2021), xix+550 pp.
ISBN: 978--1--108--82028--8.

\bibitem[Ma97]{Matsumoto1997} K. Matsumoto,
\emph{On C*-algebras associated with subshifts},
Internat. J. Math. \textbf{8} (1997), 357--374.

\bibitem[Ma10]{Matsumoto2010} K. Matsumoto, 
\emph{Orbit equivalence of topological Markov shifts and Cuntz--Krieger algebras},
Pacific J. Math. \textbf{246} (2010), 199--225.

\bibitem[Mat17a]{Matsumoto2017-circle} K.~Matsumoto,
\emph{Continuous orbit equivalence, flow equivalence of Markov shifts and circle actions on Cuntz--Krieger algebras},
Math. Z. \textbf{285} (2017), 121--141.

\bibitem[Ma17b]{Matsumoto2017-ucoe} K.~Matsumoto,
\emph{Uniformly continuous orbit equivalence of Markov shifts and gauge actions on Cuntz--Krieger algebras},
Proc. Amer. Math. Soc. \textbf{145} (2017), no. 3, 1131--1140.

\bibitem[Ma20]{Matsumoto2020} K. Matsumoto,
\emph{A groupoid approach to C*-algebras associated with $\lambda$-graph systems and continuous orbit equvalence of subshifts},
Dyn. Syst. \textbf{35} (2020), no. 3, 398--429.

\bibitem[Ma]{Matsumoto2020-normal} K. Matsumoto,
\emph{Simple purely infinite C*-algebras associated with normal subshifts},
(arXiv:2003.11711v2 [math.OA]).
 
\bibitem[Mat21]{Matsumoto2021} K.~Matsumoto,
\emph{One-sided topological conjugacy of normal subshifts and gauge actions on the associated C*-algebras},
Dyn. Syst. \textbf{36} (2021), no. 4, 586--607.

\bibitem[Mat22]{Matsumoto2022} K.~Matsumoto,
\emph{On one-sided topological conjugacy of topological Markov shifts and gauge actions on Cuntz--Krieger algebras},
Ergodic Theory Dynam. Systems \textbf{42} (2022), no. 8, 2575--2582.

\bibitem[MM14]{Matsumoto-Matui2014} K. Matsumoto and H. Matui,
\emph{Continuous orbit equivalence of topological Markov shifts and Cuntz--Krieger algebras},
Kyoto J. Math. \textbf{54} (2014), 863--877.

\bibitem[Ma06]{Matui2006} H.~Matui,
\emph{Some remarks on topological full groups of Cantor minimal systems},
Internat. J. Math. \textbf{17} (2006), no. 2, 231--251.

\bibitem[Ma12]{Matui2012} H. Matui,
\emph{Homology and topological full groups of \'etale groupoids on totally disconnected spaces},
Proc. Lond. Math. Soc. (3) (2012) Doi:{10.1112/plms/pdr029}, 27--56.

\bibitem[Ma15]{Matui2015} H. Matui,
\emph{Topological full groups of one-sided shifts of finite type},
J. reine angew. Math. \textbf{705} (2015), 35--84.

\bibitem[MvN43]{Murray-Neumann1943} F.J.~Murray and J.~von Neumann
\emph{On rings of operators},
IV. Ann. of Math. (2) \textbf{44} (1943), 716--808.

\bibitem[Na86]{Nasu1986} M.~Nasu,
\emph{Topological conjugacy for sofic systems},
Ergodic Theory Dynam. Systems \textbf{6} (1986), no. 2, 265--280.

\bibitem[vN31]{Neumann1931} J.~von Neumann,
\emph{Proof of the quasi-ergodic hypothesis},
Proc. N.A.S. \textbf{18} (1) (1931), 70--82.

\bibitem[Pa66]{Parry1966} W.~Parry,
\emph{Symbolic dynamics and transformations of the unit interval},
Trans. Amer. Math. Soc. \textbf{122} (1966), 368--378.

\bibitem[PS75]{Parry-Sullivan} W. Parry and D. Sullivan
\emph{A topological invariant of flows on 1-dimensional spaces},
Topology \textbf{14} (1975), no. 4, 297--299.

\bibitem[Ra05]{Raeburn2005} I.~Raeburn,
\emph{Graph algebras},
CBMS Regional Conference Series in Mathematics, 103. 
Published for the Conference Board of the Mathematical Sciences, Washington, DC; by the American Mathematical Society, Providence, RI (2005).

\bibitem[Re80]{Renault-thesis} J.~Renault,
\emph{A groupoid approach to C*-algebras}, 
Lecture Notes in Mathematics \textbf{793}, Springer, Berlin (1980), MR 584266.

\bibitem[Re08]{Renault2008} J. Renault,
\emph{Cartan subalgebras in C*-algebras},
Irish Math. Soc. Bull. \textbf{61} (2008), 29--63.

\bibitem[R\o 95]{Rordam1995} M. R\o rdam,
\emph{Classification of Cuntz-Krieger algebras},
K-Theory \textbf{9} (1995), no. 1, 31--58.

\bibitem[Sa98]{Samuel1998} J.~Samuel,
\emph{C*-algebras of sofic shifts},
PhD thesis (1998), University of Victoria.

\bibitem[Si20]{Sims-notes} A. Sims,
\emph{Hausdorff \'etale groupoids and their C*-algebras},
in ``Operator algebras and dynamics: groupoids, crossed products and {R}okhlin dimension'' in 
\emph{Advanced Courses in Mathematics. CRM Barcelona}, Birk\"auser, 2020.

\bibitem[SW16]{Sims-Williams} A. Sims and D.P. Williams,
\emph{The primitive ideals of some \'etale groupoid C*-algebras}, 
Algebras and Representation Theory \textbf{19} (2016), 255--276.

\bibitem[We73]{Weiss1973} B.~Weiss,
\emph{subshifts of finite type and sofic systems},
Monatsh. Math. \textbf{77} (1973), 462--474.

\bibitem[Wi73]{Williams1973} R.F. Williams,
\emph{Classification of subshifts of finite type},
Ann.\ of Math. (2) \textbf{98} (1973), 120--153; errata, ibid. (2) \textbf{99} (1974), 380--381.

\end{thebibliography}
\end{document}